\documentclass[a4paper]{amsart}

\usepackage[a4paper,width={16cm},left=2.5cm,bottom=3cm, top=3cm]{geometry}
\usepackage{enumerate}

\usepackage[utf8]{inputenc}
\usepackage[english]{babel}

\usepackage{a4wide}
\usepackage{xcolor}
\usepackage{amsmath}
\usepackage{amssymb}
\usepackage{amsfonts}
\usepackage{amsthm,graphicx}

\usepackage{hyperref}
\usepackage{mathtools}
\mathtoolsset{showonlyrefs}

\newtheorem{theorem}{Theorem}[section]
\newtheorem{lemma}[theorem]{Lemma}
\newtheorem{corollary}[theorem]{Corollary}

\newtheorem{proposition}[theorem]{Proposition}
\theoremstyle{definition}

\newtheorem{remark} [theorem] {Remark}

\newcommand{\mc}{\mathcal}
\newcommand{\mb}{\mathbb}

\newcommand{\la}{\lambda}
\newcommand{\norm}[1]{\left\lVert#1\right\rVert}
\newcommand{\pd}[2]{\frac{\partial#1}{\partial#2}}

\newcommand{\R}{\mb{R}}

\newcommand{\e}{\varepsilon}

\newcommand{\Tr}{\mathop{\rm{Tr}}}
\newcommand{\dive}{\mathop{\rm{div}}}

\newcommand{\df}[4]{\ensuremath\sideset{_{#1}}{_{#4}}{\mathop{\left\langle #2, #3 \right\rangle}}}
\newcommand{\ps}[3]{\left( #2, #3 \right)_{#1}}

\begin{document}
\title[On the fractional powers of a Schr\"odinger operator with a Hardy-type potential]
{On the fractional powers of a Schr\"odinger operator with a Hardy-type potential}

\date{02/21/2023}

\author{ Giovanni Siclari}

\address[ G. Siclari]{Dipartimento di Matematica e Applicazioni
\newline\indent
Universit\`a degli Studi di Milano - Bicocca
\newline\indent
Via Cozzi 55, 20125, Milano, Italy}
\email{g.siclari2@campus.unimib.it}

\thanks{
G. Siclari is partially supported by the GNAMPA-INdAM 2022 grant ``Questioni di esistenza e unicit\`a per problemi non locali con potenziali''. 
}

\begin{abstract}
\noindent  
Strong unique	continuation properties and a classification of the asymptotic profiles are established for the fractional powers of a Schr\"odinger operator  with a Hardy-type potential, by  means of an  Almgren monotonicity formula combined with a blow-up analysis.
\end{abstract}

\maketitle

\section{Introduction} \label{sec_introduction}
This paper  deals with  the fractional powers of the operator 
\begin{equation}\label{def_operator_integer}
L_{\alpha,k}u:=-\Delta u -  \frac{\alpha }{|x|_{k}^2}u
\end{equation}
on a connected bounded Lipschitz domain $\Omega \subset \R^N$ with $N\ge 3$ and  $0 \in \Omega$, where

\begin{equation}\label{hp_alpha}
|x|_{k}^2=\sum_{i=1}^k x_i^2 \quad \text{ and }  \quad \alpha \in \left(-\infty,\left(\frac{k-2}{2}\right)^2\right)
\end{equation}
for any $k\in \{3,\dots, N\}$. If $k=N$ we will simply write $|x|$ for $|x|_N$.

The operator  $L_{\alpha,{k}}$ is an elliptic operator with a homogenous potential with a singular set of dimension $N-k$.
In view of Hardy-Maz'ja-type inequalities, see Section \ref{sec_functional_setting_and_main_result}, the operator $L_{\alpha,k}$ has a discrete spectrum on $H^1_0(\Omega)$. Hence the fractional  powers 
$L^s_{\alpha,k}$ of $L_{\alpha,k}$ with $s \in (0,1)$ can be defined in a spectral sense, see for example \cite{ST}.
In the particular case  $\alpha=0$ the operator $L_{\alpha,k}^s$ reduces to the spectral fractional Laplacian $(-\Delta)^s$ which has been intensely studied in the literature, 
see for example \cite{AD,LAPGGM} and the references within. 

We will give a more precise  definition of $L^s_{\alpha,k}$ in Section \ref{sec_functional_setting_and_main_result} 
since, to the best of the author's knowledge,  the operator  $L^s_{\alpha,k}$ has not been considered before in the literature with  $\alpha \neq 0$ in a bounded domain.  
In the whole space $\R^N$ the fractional powers of $L_{\alpha,N}$ have already been defined by the means of spectral theory, see  \cite{FHhardy}. In \cite{FHhardy} 
generalised and reversed Hardy types inequalities  have been obtained for $L_{\alpha,N}^s$ using semigroup theory and estimates on the corresponding heat kernel.

We will focus on the validity of a  unique continuation principle from the singular point $0$ for solutions of linear  equations involving  the operator $L_{\alpha,k}^s$.
We  are interested in the equation 
\begin{equation}\label{eq_fractional_hardy}
	L_{\alpha,k}^su=gu \quad \text{ in }\Omega
\end{equation}
where the potential $g$ satisfies
\begin{equation}\label{hp_g}
	\begin{cases}
		g \in W^{1,\infty}_{loc}(\Omega \setminus\{0\}), \\
		|g(x)|+ |x \cdot \nabla g (x)| \le C_g |x|^{-2s+\e},  \quad \text{ for a.e. } x \in \Omega,
	\end{cases}
\end{equation}
for some positive constant $C_g>0$ and $\e \in (0,1)$.
We will classify  the asymptotic profiles in $0$ of  solutions    of \eqref{eq_fractional_hardy} in a suitable 
weak sense, and obtain a strong unique continuation property from $0$, see Theorem \ref{theorem_blow_up_not_precise}, Theorem  \ref{theorem_blow_up_precise}, and 
Corollary \ref{corollary_unique_continuation} for a precise statements of our results. In particular we will prove that  the asymptotic profile of $u$ in $0$ is an homogenous function. We will also characterize the possible orders of homogeneity, which have a non-trivial dependence on the singular potential $\alpha |x|_k^{-2}$,  see Theorem \ref{theorem_blow_up_not_precise}.

For the  restricted  fractional Laplacian with a Hardy-type potential, under similar assumptions on the potential $g$  and with a non-linear term, a  
complete classification of the possible asymptotic profiles and a unique continuation property from $0$ have been obtained in \cite{FF}. 
The asymptotic behaviour of the spectral fractional Laplacian with a Hardy-type potential is identical  since the equivalent problem obtained with a Caffarelli-Silvestre  extension procedure is the same locally.
The  restricted  fractional Laplacian with a Hardy-type 
potential has been intensively  studied in the literature, see for example \cite{F,BKFTJTPD,AIP,FW,DFO} and the references within.

If $k=N$, it is interesting   to compare our results with \cite{FF}, in particular the minimum  order of homogeneity of the asymptotics profiles, see \eqref{eq_first_eigenvlaue_k=N}, Theorem 
\ref{theorem_blow_up_not_precise} and \cite[Proposition 2.3]{FF}. In  our cases is possible to compute it explicitly, while for the restricted fractional Laplacian only a more implicit expression is 
available.   

Similar results in the classical case, that is $s=1$, in the much more general situation of multiple potentials, including cylindrical and multi-body ones, and with the presence of a non linear term, 
have been obtained in \cite{FFT}. Furthermore in \cite{FFT} the authors also studied  regularity properties of the solutions by   means of a  Brezis-Kato argument and obtained pointwise estimates.

To study   unique continuation properties from $0$ for solutions of \eqref{eq_fractional_hardy} we  start by defining  a precise functional setting for \eqref{eq_fractional_hardy} by  means of Interpolation Theory. Furthermore our approach is based on an Almgren type monotonicity formula combined with a blow-up argument. 
Since this approach is local in nature, we need a suitable extension result to localise the problem, see Theorem \ref{theorem_extension} and also   \cite{CDDS,CS,ST}.
We will also need a Pohozaev type identity to develop  a monotonicity formula.
The singularity of the Hardy type potential $\alpha |x|_k^{-2}$, the assumptions   \eqref{hp_g} on $g$ and the singularity or degeneracy of the Muckenhoupt weight $y^{1-2s}$ in the 
hyperplane $\R^n \times\{0\}$ cause an eventual lack of regularity for solutions to the extended problem.
We  overcame this issue by means  of an approximation procedure based on the Implicit Function Theorem and  the ideas contained in  \cite{FSlincei}.

The paper is organized as follows. In Section \ref{sec_functional_setting_and_main_result} we provide the precise functional setting for  \eqref{eq_fractional_hardy} and state our main results. In 
Section \ref{sec_inequalities_and_extension} we prove the extension Theorem  \ref{theorem_extension}, study an eigenvalue problem on a hemisphere, which will turn out to be correlated to the asymptotic 
profiles of weak solutions of \eqref{eq_fractional_hardy}, and discuss some useful inequalities. In Section \ref{sec_Pohozaev_identity}  we  prove a Pohozaev type identity. In Section 
\ref{sec_monotonicity} we  develop a monotonicity formula for the extend problem while in Section \ref{sec_the_blow_up_analysis} we carry out the blow-up argument and prove our main results. 
Finally in Section  \ref{sec_computation_eigenvalue}  we  compute the first eigenvalue of the problem studied in \ref{sec_inequalities_and_extension} while in Appendix \ref{sec_appendix_A} we provide some further details on the functional setting for   equation \eqref{eq_fractional_hardy}  which will be introduced in Section \ref{sec_functional_setting_and_main_result}.

\section{Functional Setting and Main Results} \label{sec_functional_setting_and_main_result}
Since we deal with singular potentials of the form $\alpha|x|_{k}^{-2}$, Hardy-type inequalities with optimal constants  are fundamental to study the positivity of $L_{\alpha,k}$ on $H^1_0(\Omega)$. In the case $k=N$ it is well known that 
\begin{equation}\label{ineq_hardy_N}
\int_{\R^N}\frac{\phi^2}{|x|^2}\, dx \le \left(\frac{2}{N-2}\right)^2\int_{\R^N} |\nabla \phi|^2 \, dx \quad \text { for any } \phi \in C^{\infty}_c(\R^N),
\end{equation}
and that $\left(\dfrac{2}{N-2}\right)^2$ is the optimal constant.
A similar result also holds for cylindrical potential, more precisely for   any $k \in \{3,\dots,N\}$
\begin{equation}\label{ineq_hardy_k}
\int_{\R^N}\frac{\phi^2}{|x|_k^2}\, dx \le \left(\frac{2}{k-2}\right)^2\int_{\R^N} |\nabla \phi^2| \, dx \quad \text { for any } \phi \in C^{\infty}_c(\R^N),
\end{equation}
see \cite[ Subsection 2.1.6, Corollary 3]{M} or \cite{BT}. Furthermore   $\left(\frac{2}{k-2}\right)^2$ is the optimal constant as conjectured in \cite{BT} and proved in  \cite{SSW}.

Let us consider the eigenvalue problem 
\begin{equation}\label{prob_eigenvalue_omega}
\begin{cases}
L_{\alpha,k}u=\mu u, &\text{ in } \Omega,\\
u=0, &\text{ on } \partial \Omega.
\end{cases} 
\end{equation}
We say that $\mu$ is an eigenvalue of \eqref{prob_eigenvalue_omega} if there exists $Y \in H_0^1(\Omega) \setminus\{0\}$ such that 
\begin{equation}\label{eq_eigenvalue_omega}
\int_{\Omega}\nabla Y \cdot \nabla v \, dx -\int_{\Omega}\frac{\alpha }{|x|_{k}^2}Yv \, dx= \mu \int_{\Omega} Y v \, dx, \quad \text{ for any } v \in H^1_0(\Omega). 
\end{equation}
Thanks to \eqref{hp_alpha} and \eqref{ineq_hardy_k}, for any $k \in \{3,\cdots, N\}$ the energy functional
\begin{equation}\label{def_J}
J_{\alpha,k}(u):=\int_{\Omega}|\nabla u|^2 \, dx -\int_{\Omega}\frac{\alpha }{|x|_{k}^2}u^2 \, dx
\end{equation}
is coercive on $H^1_0(\Omega)$ and so by the Spectral Theorem  the set of the  eigenvalues of \eqref{prob_eigenvalue_omega} is a non-decreasing, positive, diverging  sequence  
$\{\mu_{\alpha,k,n}\}_{n\in \mb{N}\setminus \{0\}}$(we repeat  each eigenvalue according to its multiplicity). Furthermore there exists an orthonormal basis $ \{Y_{\alpha,k,n}\}_{n\in \mb{N}\setminus 
\{0\}}$ of $L^2(\Omega)$ made of corresponding eigenfunctions. Since the first eigenfunction does not change sign, it is not restrictive to suppose that $Y_{\alpha,k,1}$ is positive.

For any Hilbert space $X$ let  $(v_1,v_2)_{X}$ be the scalar product on $X$. Furthermore let
\begin{equation}\label{def_vn}
v_n:=(v,Y_{\alpha,k,n})_{L^2(\Omega)} \quad  \text{ for any } v \in L^2(\Omega).
\end{equation}

\begin{remark}\label{remark_equivalence_norm}
In view of \eqref{ineq_hardy_k}, $\norm{v}_{\alpha,k}:=\left(J_{\alpha,k}(v)\right)^\frac{1}{2}$ is a norm on $H^1_0(\Omega)$ equivalent to the usual norm 
$\norm{v}_{H^1_0(\Omega)}:=\left(\int_{\Omega}|\nabla v|^2 \, dx\right)^\frac{1}{2}$. 
The scalar product associated to $\norm{\cdot}_{\alpha,k}$ is given by 
\begin{equation}\label{def_scalar_product_alphak}
(v,w)_{\alpha,k}:=\int_{\Omega} \nabla v  \cdot\nabla w-\frac{\alpha}{|x|^2_k}v w\, dx.
\end{equation}
By \eqref{eq_eigenvalue_omega}, $\{Y_{\alpha,k,n}/\sqrt{\mu_{\alpha,k,n}}\}_{n\in \mb{N}\setminus \{0\}}$ is an orthonormal basis of  $H^1_0(\Omega)$ with respect to the norm $\norm{\cdot}_{\alpha,k}$ and for any $ v,w\in H^{1}_0(\Omega)$
\begin{equation}\label{scalar_product_alphak_as_series}
(v,w)_{\alpha,k}=\sum_{n=1}^\infty\mu_{\alpha,k,n} v_n w_n,
\end{equation}
where  $v_n$ and $w_n$ are as in \eqref{def_vn}.	
\end{remark}

Let  us consider the functional space
\begin{equation*}
\mb{H}_{\alpha,k}^s(\Omega):=\left\{v \in L^2(\Omega):	\sum_{n=1}^\infty\mu^s_{\alpha,k,n}v_n^2	<+\infty\right\}
\end{equation*}
which is  a Hilbert space with respect to the scalar product
\begin{equation}\label{fract-lapla-domain-scalar-product}
(v,w)_{\mb{H}_{\alpha,k}^s(\Omega)}:=\sum_{n=1}^\infty\mu^s_{\alpha,k,n}	v_nw_n,	\quad \text{ for any } v,w \in \mb{H}_{\alpha,k}^s(\Omega).
\end{equation}
For any $j \in \mb{N}\setminus\{0\}$, and $v \in L^2(\Omega)$ it is clear that  $\sum_{n=1}^j\mu^s_{\alpha,k,n} v_n Y_{\alpha,k,n} \in L^2(\Omega)$ and that it can be identified with the  element of 
the dual space $(\mb{H}_{\alpha,k}^s(\Omega))^*$ acting on $u \in\mb{H}_{\alpha,k}^s(\Omega)$ as 
\begin{equation*}
\df{(\mb{H}_{\alpha,k}^s(\Omega))^*}{\sum_{n=1}^j\mu^s_{\alpha,k,n} v_n Y_{\alpha,k,n}}{u}{\mb{H}_{\alpha,k}^s(\Omega)}:=\left(\sum_{n=1}^j\mu^s_{\alpha,k,n} v_n Y_{\alpha,k,n},u\right)_{L^2(\Omega)}
=\sum_{n=1}^j\mu^s_{\alpha,k,n} v_n u_n.
\end{equation*}
It is easy to see  that, if $v\in \mb{H}_{\alpha,k}^s(\Omega)$, then the series	$\sum_{n=1}^\infty\mu^s_{\alpha,k,n} v_n Y_{\alpha,k,n}$
converges in the dual space $(\mb{H}_{\alpha,k}^s(\Omega))^*$ to some $F\in (\mb{H}_{\alpha,k}^s(\Omega))^*$ such that
\begin{equation*}
\df{(\mb{H}_{\alpha,k}^s(\Omega))^*}{F}{Y_{\alpha,k,n}}{\mb{H}_{\alpha,k}^s(\Omega)}=\mu^s_{\alpha,k,n}
v_n \quad \text{ for any } n \in \mb{N}\setminus\{0\}.		
\end{equation*} 
It follows that, for every 	$v\in \mb{H}_{\alpha,k}^s(\Omega)$, we can define the fractional $s$-power of the operator $L_{\alpha,k}$ as  
\begin{equation}\label{def_fractional_L}
L_{\alpha,k}^s v := \sum_{n=1}^\infty\mu^s_{\alpha,k,n}v_nY_{\alpha,k,n} \in (\mb{H}_{\alpha,k}^s(\Omega))^*.
\end{equation}
More precisely, the operator $L_{\alpha,k}^s $ is the Rietz isomorphism between $\mb{H}_{\alpha,k}^s(\Omega)$ endowed with the scalar product \eqref{fract-lapla-domain-scalar-product} and its dual 
space $(\mb{H}_{\alpha,k}^s(\Omega))^*$, that is 
\begin{equation}\label{property_fractional_L_Rietz}
\df{(\mb{H}_{\alpha,k}^s(\Omega))^*}{L_{\alpha,k}^s v_1}{v_2}{\mb{H}_{\alpha,k}^s(\Omega)}=
(v_1,v_2)_{\mb{H}_{\alpha,k}^s(\Omega)}	\quad \text{ for all } v_1,v_2 \in\mb{H}_{\alpha,k}^s(\Omega).
\end{equation}
A similar definition for the spectral fractional Laplacian, that is the operator $L_{0,N}$, was given in \cite{CDDS} and in  \cite{DFSboundary}.

We would like to characterize the space $\mb{H}_{\alpha,k}^s(\Omega)$ more explicitly. To this end, let $H^s(\Omega)$ be the usual fractional Sobolev space $W^{s,2}(\Omega)$, $H_0^s(\Omega)$ the closure of $C^\infty_c(\Omega)$ in $H^s(\Omega)$, and let 
\begin{equation*}
H_{00}^{1/2}(\Omega):=\left\{u \in H_0^{\frac12}(\Omega):\int_{\Omega}\frac{u^2(x)}{d(x,\partial \Omega)}\, dx <+\infty\right\},	
\end{equation*}
endowed with the norm 
\begin{equation}\label{def_norm_H00}
\norm{v}_{H_{00}^{1/2}(\Omega)}:= \norm{v}_{H^{1/2}(\Omega)}+\left(\int_{\Omega}\frac{v^2(x)}{d(x,\partial \Omega)}\, dx\right)^\frac{1}{2},
\end{equation} 
where $d(x,\partial \Omega):=\inf\{|x-y|:y \in \partial \Omega\}$.
For any $s \in (0,1)$ let 
\begin{equation}\label{def_Hs}
\mb{H}^s(\Omega):=
\begin{cases}
H^s_0(\Omega), &\text{if } s \in(0,1)\setminus\{\frac{1}{2}\}, \\
H_{00}^{1/2}(\Omega), &\text{if } s =\frac{1}{2}.
\end{cases}	 
\end{equation}
We also note  that $H^s(\Omega)=H^s_0(\Omega)$ if and only if $s \in (0,\frac{1}{2}]$, see  \cite[Theorem 11.1]{LM1}. 
In Appendix \ref{sec_appendix_A} we will prove the following Proposition by means of Interpolation Theory.
\begin{proposition}\label{prop_Halphak_Hs}
For any $k \in \{3,\dots,N\}$, $s \in (0,1)$ and $\alpha$ as in \eqref{hp_alpha} 
\begin{equation}\label{eq_Halphak_Hs}
\mb{H}_{\alpha,k}^s(\Omega)=(L^2(\Omega),H^1_0(\Omega))_{s,2}=\mb{H}^s(\Omega),
\end{equation}
with equivalent norms.
\end{proposition}

Let for any measurable function $v:\Omega \to \R$,
\begin{equation}\label{prop_ineq_hardy_Hs:1}
\tilde{v}(x):=
\begin{cases}
v(x), &\text{ if } x \in \Omega, \\
0, &\text{ if } x \in \R^N \setminus \Omega.
\end{cases}
\end{equation}
Then from  \cite[Proposition B.1]{BLP} in the case $s\neq \frac{1}{2}$ and from the proof of \cite[Proposition B.1]{BLP} and \eqref{def_norm_H00} if $s=\frac12$ we deduce the following result.
\begin{proposition} \label{prop_extension_to_0}
There exists a constant $C_{N,s,\Omega}$ such that   
\begin{equation}\label{ineq_extension_to_0}
\norm{\tilde{v}}_{H^s(\R^n)}\le C_{N,s,\Omega}\norm{v}_{\mb{H}^s(\Omega)}
\end{equation}
for any $v \in \mb{H}^s(\Omega)$.
\end{proposition}

\begin{proposition} \label{prop_ineq_hardy_Hs}
There exists a constant $K_{N,s,\Omega}$ such that for any $v \in \mb{H}^s(\Omega)$ 
\begin{equation}\label{ineq_hardy_Hs}
\int_{\Omega} \frac{v^2(x)}{|x|^{2s}} \, dx \le K_{N,s,\Omega}\norm{v}_{\mb{H}^s(\Omega)}^2.
\end{equation}
\end{proposition}
\begin{proof}
The  following Hardy-type inequality due to Herbst \cite{H} 
\begin{equation}\label{ineq_hardy-frac-Herbst}
2^{2s}\frac{\Gamma^2\left(\frac{N+2s}{4}\right)}{\Gamma^2\left(\frac{N-2s}{4}\right)}\int_{\R^{N}} \frac{v^2(x)}{|x|^{2s}} \, dx \le \int_{\R^N} |\xi|^{2s} |\hat{u}(\xi)|^2 d \xi,
\end{equation}
where $\hat{u}$ is the Fourier transform of $u$, holds for any $v \in H^s(\R^N)$. Then  \eqref{ineq_hardy_Hs} follows from \eqref{ineq_extension_to_0}.
\end{proof}

By Proposition \ref{prop_Halphak_Hs}, we can define a weak solution  to \eqref{eq_fractional_hardy} as a function $u \in \mb{H}^s(\Omega)$ such that  
\begin{equation}\label{eq_weak_formulation_hardy_operator}
\df{(\mb{H}_{\alpha,k}^s(\Omega))^*}{L_{\alpha,k}^s u}{\phi}{\mb{H}_{\alpha,k}^s(\Omega)}=\int_{\Omega} gu  \phi \, dx, \quad \text{ for any } \phi \in C^{\infty}_c(\Omega).
\end{equation}
Thanks to \eqref{hp_g}, \eqref{ineq_hardy_Hs} and the H\"older inequality, the right hand side of \eqref{eq_weak_formulation_hardy_operator} is well defined, that is it belongs to $(\mb{H}^s(\Omega))^*$ as a 
linear functional of $\phi$.

Given the local nature of the Almgren monotonicity formula we need to localize the problem by means of an extension procedure 
in the spirt of \cite{CDDS} or \cite{CS}, see also \cite[Section 3.1]{ST}.
Let us set some notation first. Let  $z=(x,y) \in \R^N \times[0,+\infty)$ be the total variable in $\R^{N+1}_+:= \R^N \times[0,+\infty)$
and let 
\begin{equation}\label{def_C}
C:= \Omega \times (0,+\infty), \quad \partial C_L:= \partial \Omega \times (0,+\infty).
\end{equation}
For any open set $E \subseteq \R^{N+1}_+$  and any  $\phi\in C^\infty(\overline {E})$ we define
\begin{equation}\label{def_norm_H1t}
\norm{\phi}_{H^1(E,y^{1-2s})}:=\left(\int_{E} y^{1-2s}(\phi^2+|\nabla \phi|^2)\, dz\right)^{\frac{1}{2}} 
\end{equation}	
and $H^1(E,y^{1-2s})$ as the completion  of $C^\infty(\overline{E})$ with respect to the norm defined in \eqref{def_norm_H1t}. 
Thanks to  \cite[Theorem 11.11, Theorem 11.2, 11.12 Remarks (iii)]{K}, for any Lipschitz subset $E$ of  $\R^{N+1}_+$, the space $H^1(E,y^{1-2s})$ can be explicitly characterized as  
\begin{equation}\label{def_space_H1t}
H^1(E,y^{1-2s})=\left\{V \in W^{1,1}_{\rm loc}(E):\int_{E} y^{1-2s} (V^2+|\nabla V|^2)\, dz< +\infty\right\}.
\end{equation}

\begin{proposition}\label{prop_ineq_hardy_extened}
For any $\phi \in C^\infty_c(\R^{N} \times [0,+\infty) )$ and any $k \in \{3,\dots,N\}$ 
\begin{equation}\label{ineq_hardy_extened}
\int_{\R^{N+1}_+ } y^{1-2s} \frac{\phi^2}{|x|^2_k} dz \le \left(\frac{2}{k-2}\right)^2 \int_{\R^{N+1}_+ } y^{1-2s}|\nabla_x \phi|^2 \, dz,
\end{equation}
where $\nabla_x$ is the gradient respect to the first $N$ variables.
\end{proposition}
\begin{proof}
Let  $\phi \in C^\infty_c(\Omega \times [0,+\infty) )$ and $k \in \{3,\dots,N\}$.
Then  $\phi(\cdot,y)\in C_c^{\infty}(\Omega)$ for any $ y \in [0,\infty)$ and so multiplying by $y^{1-2s}$ and integrating over $(0,\infty)$ we deduce   \eqref{ineq_hardy_extened} from \eqref{ineq_hardy_k}.
\end{proof}

Let
\begin{equation}\label{def_H10L}
H^1_{0,L}(C,y^{1-2s}):=\left\{V \in H^1(C,y^{1-2s}): V=0 \text{ on } \partial C_L\right\}.
\end{equation}
The condition $V=0$ on $\partial C_L$ is meant in a classical trace sense. Indeed    the weight $y^{1-2s}$ is smooth, bounded and strictly positive on $\Omega \times [y_1,y_2]$ for any $0<y_1<y_2<+\infty$, and so we can use classical trace theory for the space $H^1(\Omega \times (y_1,y_2))$ for any $0<y_1<y_2<+\infty$.

From  \cite[Proposition 2.1]{CDDS} and \cite[Proposition 2.1, Lemma 2.6]{CT} we deduce the following result.
\begin{proposition} \label{prop_trace} 
There exists  a linear and continuous trace operator 
\begin{equation}\label{def_Tr_omega}
\Tr:H^1_{0,L}(C,y^{1-2s}) \to \mb{H}^s(\Omega)
\end{equation}
which is also surjective.
\end{proposition}

 See Section \ref{sec_inequalities_and_extension} for a proof of the  following next extension theorem,.
\begin{theorem}\label{theorem_extension}
Let $ v \in \mb{H}^s(\Omega)$, $k \in \{3,\dots,N\}$ and $\alpha$ as in \eqref{hp_alpha}. Then there exists a unique function $V \in H_{0,L}^1(C,y^{1-2s})$ such that $V$ weakly solves the problem
\begin{equation}\label{prob_extension}
\begin{cases}
-\dive(y^{1-2s}\nabla V)= y^{1-2s} \frac{\alpha}{|x|_k^2}V, \quad &\text{ in } C,\\
\Tr(V)=v, \quad &\text{ on } \Omega,\\
-\lim_{y \to 0^+}y^{1-2s}\pd{V}{y}=c_{N,s}L^s_{k,\alpha }v, \quad &\text{ on } \Omega,
\end{cases}
\end{equation}
where $c_{N,s}>0$ is a constant depending only on $N$ and $s$, in the sense that 
\begin{equation}\label{eq_extension}
\int_{C} y^{1-2s} \nabla V\cdot\nabla \phi \, dz-\int_{C} y^{1-2s}\frac{\alpha}{|x|_k^2}V\phi  \, dz
=c_{N,s}\df{(\mb{H}_{\alpha,k}^s(\Omega))^*}{L_{\alpha,k}^s v}{\phi(\cdot,0)}{\mb{H}_{\alpha,k}^s(\Omega)}
\end{equation}
for any $\phi \in C_c^{\infty}(\Omega \times [0,+\infty))$. Furthermore 
\begin{equation}\label{eq_norm_V_norm_v}
\int_{C}y^{1-2s} |\nabla V(x,y)|^2 \, dz-\int_{C}y^{1-2s} \frac{\alpha}{|x|^2_k}V^2 \, dz = c_{N,s} \norm{v}^2_{\mb{H}_{\alpha,k}^s(\Omega)}
\end{equation} 
and $V$ is the only solution to  the minimization problem 
\begin{equation}\label{def_min_prob}
\inf\left\{\int_{C}y^{1-2s} \left(|\nabla W|^2-\frac{\alpha}{|x|^2_k}w^2\right) \,dz: W \in H_{0,L}^1(C,y^{1-2s}) \text{ and } \Tr(W)=v\right\}.
\end{equation}
\end{theorem}

From Theorem  \ref{theorem_extension} we deduce the following corollary.
\begin{corollary}\label{corollary_extension}
Let $u \in \mb{H}^s(\Omega)$ be a solution of \eqref{eq_weak_formulation_hardy_operator}.  Then there exists  a unique   $U \in H_{0,L}^1(C,y^{1-2s})$  such that 
\begin{equation}\label{prob_extension_with_equation}
\begin{cases}
-\dive(y^{1-2s}\nabla U)= y^{1-2s} \frac{\alpha}{|x|_k^2}U, \quad &\text{ in } C,\\
\Tr(U)=u, \quad &\text{ on } \Omega,\\
-\lim_{y \to 0^+}y^{1-2s}\pd{U}{y}=c_{N,s}gu, \quad &\text{ on } \Omega,
\end{cases}
\end{equation}
where $c_{N,s}>0$ is the constant depending only on $N$ and $s$ defined in Theorem \ref{theorem_extension}, in the sense that 
\begin{equation}\label{eq_extension_with_equation}
\int_{C} y^{1-2s} \nabla U\cdot\nabla \phi \, dz-\int_{C} y^{1-2s}\frac{\alpha}{|x|_k^2}U\phi  \, dz
=c_{N,s}\int_{\Omega} gu  \phi(\cdot,0) \, dx
\end{equation}
for any $\phi \in C_c^{\infty}(\Omega  \times [0,+\infty))$.
Furthermore 
\begin{equation}\label{eq_norm_U_norm_u}
\int_{C}y^{1-2s} |\nabla U(x,y)|^2 \, dz-\int_{C}y^{1-2s} \frac{\alpha}{|x|^2_k}U^2 \, dz = c_{N,s} \norm{u}^2_{\mb{H}_{\alpha,k}^s(\Omega)}=c_{N,s}\int_{\Omega} gu^2 \, dx.
\end{equation} 
\end{corollary}
Let for, any $r>0$,
\begin{align*}
&B^+_r:= \{z \in \R^{N+1}_+: |z|<r\}, \quad S_r^+:=\{z \in \R^{N+1}_+: |z|=r\},\\
&B'_r:= \{z=(x,y) \in \R^{N+1}: |x|<r, y=0\}.
\end{align*}
Let    $\theta=\frac{z}{|z|}$ for any $z\in \R^{N+1}$ and  $\theta'=(\theta_1, \dots, \theta_N)$.

The asymptotic profile of  a solution $U$ of \eqref{eq_extension_with_equation} in $0$  will turn out to be related to the following eigenvalue problem
\begin{equation}\label{prob_eigenvalue_sphere}
\begin{cases}
-\dive_{\mb{S}}(\theta_{N+1}^{1-2s}\nabla_{\mb{S}}Z)-\theta_{N+1}^{1-2s}\frac{\alpha}{|\theta|^2_k}Z=\gamma\theta_{N+1}^{1-2s} Z, & \text{ in  } \mb{S}^+,\\
-\lim\limits_{\theta_{N+1}\to 0^+}\theta_{N+1}^{1-2s} \nabla_{\mb{S}}Z \cdot \nu =0, & \text{ on } \mb{S}',
\end{cases}
\end{equation} 
where $\nu$ is the outer normal vector to $\mb{S}^+$ on $\mb{S}'$, that is $\nu=-(0,\dots, 0,1)$ and 
\begin{align*}
&\mb{S}:=\{\theta\in \R^{N+1}:|\theta|^2=1\},\\
&\mb{S}^+:=\{\theta=(\theta', \theta_{N+1}) \in \mb{S}: \theta_{N+1}>0\},\\
&\mb{S}':=\{\theta=(\theta', \theta_{N+1}) \in \mb{S}: \theta_{N+1}=0\}.
\end{align*}
We refer to  Subsection \ref{subsec_an_eigenvalue_problem_on_S+} for a variational formulation of \eqref{prob_eigenvalue_sphere}.
By classical spectral theory, see Subsection  \ref{subsec_an_eigenvalue_problem_on_S+} for further details,
the eigenvalues of \eqref{prob_eigenvalue_sphere} 
are a non-decreasing and diverging sequence $\{\gamma_{\alpha,k,n}\}_{n \in \mb{N}\setminus \{0\}}$ (we repeat each eigenvalue according to its multiplicity). 
We have the following estimate on $\gamma_{\alpha,k,1}$:
\begin{equation}\label{ineq_gammaalphak1}
\gamma_{\alpha,k,1}>-\left(\frac{N-2s}{2}\right)^2
\end{equation}
for any $k \in \{3,\dots,N\}$ and $\alpha$ as in \eqref{hp_alpha}, see Proposition \ref{prop_ineq_hardy_sphere}. 
We can actually compute $\gamma_{\alpha,k,1}$ in terms of the first eigenvalue $\eta_{\alpha,k,1}$ of the problem 
\begin{equation}\label{prob_eigenvalue_S'}
-\Delta_{\mb{S}'}\Psi-\frac{\alpha}{|\theta'|_k^2}\Psi= \eta \Psi \quad \text{ in } \mb{S}'
\end{equation} 
as 
\begin{equation}\label{eq_first_eigenvlaue}
\gamma_{\alpha,k,1}= 2(1-s)\left[\sqrt{\left(\frac{N-2}{2}\right)^2+\eta_{\alpha,k,1}}-\frac{N-2}{2}\right] +\eta_{\alpha,k,1},
\end{equation}
see Section \ref{sec_computation_eigenvalue}. In particular, if $k=N$ then $\eta_{\alpha,k,1}=-\alpha$ and so 
\begin{equation}\label{eq_first_eigenvlaue_k=N}
\gamma_{\alpha,N,1}= 2(1-s)\left[\sqrt{\left(\frac{N-2}{2}\right)^2-\alpha}-\frac{N-2}{2}\right] -\alpha.
\end{equation}
If $k=N$, we are able to obtain an explicit expression of $\gamma_{\alpha,N,1}$ for any $\alpha \in \left(-\infty,\frac{N-2}{2}\right)$. For  the restricted 
fractional Laplacian with a Hardy-type potential  it is also  possible to obtain a formula for the first eigenvalue of the corresponding problem on a hemisphere  although with a  more implicit expression, see \cite[Proposition 2.3]{FF}.

\begin{theorem}\label{theorem_blow_up_extension_not_precise}
Let $U$ be a non-trivial solution of \eqref{eq_extension_with_equation} and suppose that $g$ satisfies \eqref{hp_g}. Then there exist an eigenvalue  $\gamma_{\alpha,k,n}$ of  \eqref{prob_eigenvalue_sphere} 
and a correspondent eigenfunction $Z$ such that 
\begin{equation}\label{eq_limit_extension_blow_up_not_precise}
\la^{\frac{N-2s}{2}-\sqrt{\left(\frac{N-2s}{2}\right)^2+\gamma_{\alpha,k,n}}}U(\la z)\to |z|^{-\frac{N-2s}{2}+\sqrt{\left(\frac{N-2s}{2}\right)^2+\gamma_{\alpha,k,n}}} Z(z/|z|) 
\quad  \text { as } \la \to 0^+
\end{equation}
 strongly in $H^1(B_1^+,y^{1-2s})$.
\end{theorem}

\begin{remark}\label{remark_trace}
Let $r>0$. Thanks to \cite{LM1}   there exists a linear and continuous trace operator
\begin{equation}\label{def_trace_Hs}
\mathop{{\rm{Tr}}_{B'_r}}:H^1(B_r^+,y^{1-2s}) \to H^s(B_r').
\end{equation}
If $\overline{B'_r} \subset \Omega$, then  any function $V \in H^1(B_r^+,y^{1-2s})$  can be extended to an element $\tilde{V}$ of $H^1_{0,L}(C,y^{1-2s})$ (see \eqref{def_H10L} and \cite{C})  and  
$\mathop{{\rm{Tr}}}_{B'_r}(V)= \Tr(\tilde{V})_{|B_r'}$. Therefore with a slight abuse we will simply   use $\Tr$ instead of $\mathop{{\rm{Tr}}_{B'_r}}$ to indicate the operator $\mathop{{\rm{Tr}}_{B'_r}}$.
\end{remark}

From Remark  \ref{remark_trace} and the previous theorem we obtain the following.

\begin{theorem}\label{theorem_blow_up_not_precise}
Let $u$ be  a non-trivial solution solution of \eqref{eq_weak_formulation_hardy_operator} and suppose that $g$ satisfies \eqref{hp_g}. Then there exist an eigenvalue  $\gamma_{\alpha,k,n}$ of  \eqref{prob_eigenvalue_sphere} 
and a correspondent eigenfunction $Z$ such that 
\begin{equation}\label{eq_limit_blow_up_not_precise}
\la^{\frac{N-2s}{2}-\sqrt{\left(\frac{N-2s}{2}\right)^2+\gamma_{\alpha,k,n}}}u(\la x)\to |x|^{-\frac{N-2s}{2}+\sqrt{\left(\frac{N-2s}{2}\right)^2+\gamma_{\alpha,k,n}}} \Tr(Z(\cdot/|\cdot|))(x)
\quad  \text { as } \la \to 0^+
\end{equation}
strongly in $H^s(B_1')$.
\end{theorem}

We will also prove a more precise and complete version of  Theorem \ref{theorem_blow_up_extension_not_precise} and Theorem \ref{theorem_blow_up_not_precise} in 
Section \ref{sec_the_blow_up_analysis}, computing the coordinates of the eigenfunction $Z$  respect to a basis of the eigenspace corresponding to  $\gamma_{\alpha,k,n}$.
Furthermore we can deduce the following strong  unique continuation properties as corollaries of Theorem \ref{theorem_blow_up_extension_not_precise} and  Theorem \ref{theorem_blow_up_not_precise} respectively.
\begin{corollary}\label{corollary_unique_continuation_extention}
Let $U$ be a solution of \eqref{eq_extension_with_equation} and suppose that $g$ satisfies \eqref{hp_g}. If
\begin{equation}\label{eq_corollary_unique_continuation_extention}
U(z)=o(|z|^n)=o(|(x,y)|^n) \text{ as } x \to 0, \;y \to 0^+  \quad  \text{ for any  } n \in \mathbb{N}
\end{equation}
then $U \equiv 0$ on $\Omega \times (0,\infty)$.
\end{corollary}

\begin{corollary}\label{corollary_unique_continuation}
Let $u$ be a solution of \eqref{eq_weak_formulation_hardy_operator} and suppose that $g$ satisfies \eqref{hp_g}. If
\begin{equation}\label{eq_corollary_unique_continuation}
u(x)=o(|x|^n) \text{ as } x \to 0,  \quad  \text{ for any  } n \in \mathbb{N}
\end{equation}
then $u \equiv 0$ on $\Omega$.
\end{corollary}

\begin{remark}
We have considered equation \eqref{eq_fractional_hardy} with assumption \eqref{hp_g} on the potential $g$ for the sake of simplicity. With simple modifications to our arguments it is also possible 
to obtain the same results for a potential  $g \in W^{\frac{N}{2s}+\e}(\Omega)$ for some $\e\in (0,1)$, see \cite[Proposition 2.3]{FSlincei} for the corresponding Pohozaev identity. Furthermore we can obtain analogous results for the more general equation 
\begin{equation}\label{eq_frac_with_2hardy}
L^s_{k,\alpha}u=\frac{\la}{|x|^{2s}}u+gu,
\end{equation}
with $\la \in \left(-\infty, 2^{2s}\frac{\Gamma^2\left(\frac{N+2s}{4}\right)}{\Gamma^2\left(\frac{N-2s}{4}\right)}\right)$ with the same approach, where $\Gamma$ is the usual $\Gamma$-function.
\end{remark}

\section{Preliminaries} \label{sec_inequalities_and_extension}
We start this section by proving Theorem \ref{theorem_extension}.

\begin{proof}[\textbf{Proof of Theorem \ref{theorem_extension}.}] 
We follow the proof of \cite[Proposition 2.1]{CDDS}. Let $v \in \mb{H}^s(\Omega)$ and consider 
\begin{equation}\label{dim_extension:1}
V(x,y):=\sum_{n=1}^\infty v_nY_{\alpha,k,n}(x)h_n(y), \quad \text{ where } v_n=\int_{\Omega} v Y_{\alpha,k,n} \, dx
\end{equation} 
and  $h_n:(0,+\infty) \to \R$ is a solution to the problem 
\begin{equation}\label{dim_extension:2}
\begin{cases}
h_n''+\frac{1-2s}{y}h_n'-\mu_{\alpha,k,n}h_n=1, \text{ on } (0,+\infty),\\
h_n(0)=1, \\
\lim_{y \to \infty } h_n(y)=0.
\end{cases}	
\end{equation}
From the proof of \cite[Proposition 2.1]{CDDS},  \eqref{dim_extension:2} admits a  unique solution $h_n$ for any $n \in \mb{N}\setminus\{0\}$ and 
\begin{equation}\label{dim_extension:3}
-\lim_{y \to 0^+}y^{1-2s}h'_n(y)= c_{N,s}\mu^s_{\alpha,k,n},
\end{equation}
for some positive constant $c_{N,s}>0$ depending only on $N$ and $s$.
Furthermore for any $y \in [0+,\infty)$ by \eqref{dim_extension:1} and  Remark \ref{remark_equivalence_norm}  
\begin{multline}\label{dim_extension:4}
\int_{\Omega} \left|\pd{V}{y}(x,y)\right|^2 \, dx+	\int_{\Omega} |\nabla_x V(x,y)|^2 \, dx
-\int_{\Omega} \frac{\alpha}{|x|^2_k}V^2(x,y) \, dx\\
=\sum_{n=1}^{\infty} v_n^2(h'_n(y))^2+\mu_{\alpha,k,n}v_n^2h_n(y)^2.
\end{multline}
Proceeding exactly as in   \cite[Proposition 2.1]{CDDS} we can show that \eqref{eq_norm_V_norm_v} holds. Hence, in view of  \eqref{ineq_hardy_extened}, $V \in H^1(C,y^{1-2s})$ and   
$\sum_{n=1}^j v_nY_{\alpha,k,n}(x)h_n(y) \to V$ in $H^1(C,y^{1-2s})$ as $j \to \infty$. In conclusion $V \in H^1_{0,L}(C,y^{1-2s})$ since $\sum_{n=1}^j v_nY_{\alpha,k,n}(x)h_n(y) \in H_{0,L}^1(C,y^{1-2s})$ for any $j \in \mathbb{N}, j \ge 1$.

In contrast to  \cite[Proposition 2.1]{CDDS}, $V$ might not be  smooth  for $y>0$  since the functions $Y_{\alpha,k,n}$ might not be smooth on $\Omega$. Then we prove that $V$ satisfies 
\eqref{prob_extension} in the weak sense given by \eqref{eq_extension}.
Let $\phi \in C_c^{\infty}(\Omega \times [0,+\infty))$. Then 
\begin{equation}\label{dim_extension:5}
\phi(x,y)=\sum_{n=1}^\infty \phi_n(y)Y_{\alpha,k,n}(x), \quad \text{ where } \phi_n(y):=\int_{\Omega} \phi(x,y) Y_{\alpha,k,n}(x) \, dx,
\end{equation}
and similarly to \eqref{dim_extension:4}
\begin{equation}\label{dim_extension:5.1}
\int_{\Omega} |\nabla \phi(x,y)|^2 \, dx	-\int_{\Omega} \frac{\alpha}{|x|^2_k}\phi^2(x,y) \, dx=\sum_{n=1}^{\infty}(\phi'_n(y))^2+\mu_{\alpha,k,n}\phi_n(y)^2.
\end{equation}
Then by \eqref{dim_extension:1} and  Remark \ref{remark_equivalence_norm}
\begin{multline}\label{dim_extension:6}
\int_{\Omega} \nabla V(x,y) \cdot \nabla \phi(x,y)\, dx-\int_{\Omega} \frac{\alpha}{|x|^2_k}V(x,y)\phi(x,y) \, dx\\
=\sum_{n=1}^{\infty} v_nh'_n(y)\phi'_n(y)+\mu_{\alpha,k,n}v_nh_n(y)\phi_n(y).
\end{multline}
Furthermore,  for any $j \in \mb{N}$, by H\"older's  inequality
\begin{multline}\label{dim_extension:7}
\left|\int_{0}^{+\infty} y^{1-2s}\left(\sum_{n=j}^{\infty} v_nh'_n(y)\phi'_n(y)+\mu_{\alpha,k,n}v_nh_n(y)\phi_n(y) \right)\, dy\right| \\
\le \frac{1}{2} \int_{0}^{+\infty} y^{1-2s}\left(\sum_{n=j}^{\infty}  v_n^2(h'_n(y))^2+\mu_{\alpha,k,n}v_n^2h_n(y)^2  \right)\, dy\\
+\frac{1}{2}  \int_{0}^{+\infty} y^{1-2s}\left(\sum_{n=j}^{\infty}  (\phi'_n(y))^2+\mu_{\alpha,k,n}\phi_n(y)^2  \right)\, dy.
\end{multline}
By \eqref{dim_extension:4}, \eqref{dim_extension:5.1}  and the Monotone Convergence Theorem we conclude that 
\begin{equation}\label{dim_extension:8}
\lim_{j \to \infty}\int_{0}^{\infty} y^{1-2s}\left(\sum_{n=j}^{\infty} v_nh'_n(y)\phi'_n(y)+\mu_{\alpha,k,n}v_nh_n(y)\phi_n(y) \right)\, dy=0. 
\end{equation}
Hence we may change the order of summation and integration  in \eqref{dim_extension:6} obtaining 
\begin{equation}\label{dim_extension:9}
\int_{C} y^{1-2s}\left(\nabla V\cdot\nabla \phi - \frac{\alpha}{|x|^2_k}V\phi\right) \, dz=\sum_{n=1}^{\infty}v_n\int_{0}^{+\infty}y^{1-2s}(h'_n(y)\phi'_n(y)+\mu_{\alpha,k,n}h_n(y)\phi_n(y)) \, dy.
\end{equation}
An integration by parts, in view of \eqref{dim_extension:2} and \eqref{dim_extension:3}, yields
\begin{equation}\label{dim_extension:10}
\int_{0}^{+\infty}y^{1-2s}(h'_n(y)\phi'_n(y)+\mu_{\alpha,k,n}h_n(y)\phi_n(y)) \, dy=c_{N,s}\mu_{\alpha,k,n}^s\phi_n(0).
\end{equation}
It follows that 
\begin{equation}\label{dim_extension:11}
\int_{C} y^{1-2s}\nabla V\cdot\nabla \phi \, dz	-\int_{C}y^{1-2s} \frac{\alpha}{|x|^2_k}V\phi \, dz=c_{N,s}\sum_{n=1}^{\infty}\mu_{\alpha,k,n}^s v_n\phi_n(0)
\end{equation}
and so we have proved \eqref{eq_extension}.
If $V_1$ and $V_2$ solve \eqref{prob_extension} then by \eqref{hp_alpha}, \eqref{eq_extension} and \eqref{ineq_hardy_extened} we deduce that 
\begin{equation}\label{dim_extension:12}
\int_{C} y^{1-2s}|\nabla (V_1-V_2)|^2  \, dz=0, \quad \text{ and } \quad \Tr(V_1-V_2)=0
\end{equation}
thus $V_1=V_2$.
Finally $V$ solves the minimizing problem \eqref{def_min_prob} in view of  \eqref{eq_extension} and a density argument. 
\end{proof}

By \cite{FF} and \cite[Theorem 19.7]{OK} we have the following result.
\begin{proposition}\label{prop_trace_Sr+}
For any $r >0$ there exists a linear and continuous trace operator 
\begin{equation}\label{eq_trace_Sr+}
\mathop{{\rm{Tr}}_{S^{+}_r}}:H^1(B_r^+,y^{1-2s}) \to L^2(B_r^+,y^{1-2s})
\end{equation}
which is also compact. 
\end{proposition}
For the sake of simplicity, we will  write $V$ instead of $\mathop{{\rm{Tr}}_{S^{+}_r}}(V)$ on $S_r^+$.

\begin{remark}\label{remark_int_nabla_with_coarea}
For any  $r>0$ and any $V,W \in H^1(B_r^+,y^{1-2s})$, thanks to the Coarea Formula,
\begin{equation}\label{eq_int_nabla_with_coarea}
\int_{B_{r}^+} \left|y^{1-2s} \nabla U\cdot\frac{z}{|z|} W \right|dz =\int_0^{r}\left(\int_{S_\rho^+} \left|y^{1-2s} \nabla U\cdot \frac{z}{\rho} W\right| dS \right) \, d\rho
\end{equation}
hence the function $f(\rho):=\int_{S_\rho^+} \left|y^{1-2s} \nabla U\cdot \frac{z}{\rho} W\right| dS$ is a well-defined element of $L^1(0,{r})$. In particular a.e. $\rho \in (0,{r})$ is a Lebesgue point of $f$.
\end{remark}

Reasoning as in  \cite[Lemma 3.1]{FF} or\cite[Proposition 3.7]{FSlincei} we can prove the following.
\begin{proposition} \label{prop_extension_with_boundary}
Let $U$ be a solution of \eqref{eq_extension_with_equation}. For a.e. $r>0$ such that $\overline{B_r'} \subset \Omega$ and any $W \in H^1(B_r^+,y^{1-2s})$ 
\begin{multline}\label{eq_extension_with_boundary}
\int_{B_r^+} y^{1-2s} \left(\nabla U\cdot\nabla W -\frac{\alpha}{|x|_k^2}UW \right)\, dz\\
=\frac{1}{r}\int_{S_r^+} y^{1-2s} \nabla U\cdot z  \,W \, dS+c_{N,s}\int_{B_r'} g \Tr(U)\Tr(W) \, dx.
\end{multline}
\end{proposition}

\subsection{An Eigenvalue Problem on $\mb{S}^+$}\label{subsec_an_eigenvalue_problem_on_S+}
In this section we provide a variational formulation of problem \eqref{prob_eigenvalue_sphere}. To this end 
we consider  the space 
\begin{equation}\label{L2S^+}
L^2(\mb{S}^+,\theta_{N+1}^{1-2s}):=\{\Psi:\mb{S}^+ \to \R \text{ measurable: } \int_{\mb{S}^+}\theta_{N+1}^{1-2s} \Psi^2 \, dS <+\infty\},
\end{equation}
and  the space $H^1(\mb{S}^+,\theta_{N+1}^{1-2s})$ defined as the completion of $C^{\infty}(\overline{\mb{S}^+})$ with respect to the norm
\begin{equation}\label{H1S+norm}
\norm{\phi}_{H^1(\mb{S}^+,\theta_{N+1}^{1-2s})}:=\left(\int_{\mb{S}^+}\theta_{N+1}^{1-2s}( \phi^2+|\nabla_{\mb{S}} \phi|^2 )\, dS\right)^{1/2},
\end{equation}
where $\nabla_{\mb{S}}$ is the Riemannian gradient respect to the standard metric on $\mb{S}$.
\begin{proposition}\label{prop_ineq_hardy_sphere}
For any $k \in \{3,\dots,N\}$ 
\begin{equation}\label{ineq_hardy_sphere}
\left(\frac{k-2}{2}\right)^2\int_{\mb{S}^+}\theta_{N+1}^{1-2s}\frac{\Psi^2}{|\theta|_k^{2}} \, dS
\le \left(\frac{N-2s}{2}\right)^2\int_{\mb{S}^+}\theta_{N+1}^{1-2s}|\Psi|^2 \, dS+\int_{\mb{S}^+}\theta_{N+1}^{1-2s}|\nabla_{\mb{S}} \Psi|^2 \, dS
\end{equation}
for any $\Psi \in H^1(\mb{S}^+,\theta_{N+1}^{1-2s})$ .
\end{proposition}
\begin{proof}
Let $\phi \in C^\infty(\overline{\mb{S}^+})$, $f \in C_c^\infty((0,+\infty))$ with $f \neq 0$, and   $V(z):=V(r \theta)=\phi(\theta)f(r)$.
From \eqref{ineq_hardy_extened}  we obtain, passing in polar coordinates,
\begin{multline}\label{dim_ineq_hardy_sphere:1}
\left(\frac{k-2}{2}\right)^2\left(\int_{0}^{\infty} r^{N-1-2s} f^2(r) \, dr \right)\left(\int_{\mb{S}^+}\theta_{N+1}^{1-2s}\frac{\phi^2}{|\theta|_k^{2}} \, dS\right) \\
\le \left(\int_{0}^{\infty} r^{N+1-2s} |f'(r)|^2 \, dr \right)\left(\int_{\mb{S}^+}\theta_{N+1}^{1-2s}\phi^2 \, dS\right)\\
+\left(\int_{0}^{\infty} r^{N-1-2s} f^2(r) \, dr \right)\left(\int_{\mb{S}^+}\theta_{N+1}^{1-2s}|\nabla_{\mb{S}} \phi|^2 \, dS\right)
\end{multline}
and so, thanks to the optimality of the classical Hardy constant, see \cite[Theorem 330]{HLP}, 
\begin{multline}\label{dim_ineq_hardy_sphere:2}
\left(\frac{k-2}{2}\right)^2\left(\int_{\mb{S}^+}\theta_{N+1}^{1-2s}\frac{\phi^2}{|\theta|_k^{2}} \, dS\right) \\
\le \inf_{f \in C_c^\infty((0,+\infty)), f \neq 0 } \frac{\int_{0}^{\infty} r^{N+1-2s} |f'(r)|^2 \, dr }{\int_{0}^{\infty} r^{N-1-2s} f(r)^2 \, dr}
\left(\int_{\mb{S}^+}\theta_{N+1}^{1-2s}\phi^2 \, dS\right)	+\int_{\mb{S}^+}\theta_{N+1}^{1-2s}|\nabla_{\mb{S}} \phi|^2 \, dS\\
=\left(\frac{N-2s}{2}\right)^2\int_{\mb{S}^+}\theta_{N+1}^{1-2s}|\phi|^2 \, dS+\int_{\mb{S}^+}\theta_{N+1}^{1-2s}|\nabla_{\mb{S}} \phi|^2 \, dS.
\end{multline}
In conclusion \eqref{ineq_hardy_sphere} follows by density.
\end{proof}

For any $k \in \{3,\dots,N\}$ and $\alpha$ as in \eqref{hp_alpha}, we say that $\gamma$ is an eigenvalue of \eqref{prob_eigenvalue_sphere} if there exists a function $Z \in  H^1(\mb{S}^+,\theta_{N+1}^{1-2s})\setminus\{0\}$ such that 
\begin{equation}\label{eq_egienvalue_sphere}
\int_{\mb{S}^+}\theta_{N+1}^{1-2s}\nabla_{\mb{S}} Z\cdot \nabla_{\mb{S}} \Psi \, dS-\int_{\mb{S}^+}\theta_{N+1}^{1-2s}\frac{\alpha }{|\theta|_k^{2}}Z\Psi  \, dS
= \gamma\int_{\mb{S}^+}\theta_{N+1}^{1-2s}Z\Psi  \, dS,
\end{equation}
for any $\Psi \in H^1(\mb{S}^+,\theta_{N+1}^{1-2s})$. By \eqref{hp_alpha}, \eqref{ineq_hardy_sphere}, the Spectral Theorem, and the compactness of the  embedding 
$H^1(\mb{S^+},\theta_{N+1}^{1-2s}) \hookrightarrow L^2(\mb{S^+},\theta_{N+1}^{1-2s})$ (see \cite[Theorem 19.7]{OK}) the eigenvalues of \eqref{prob_eigenvalue_sphere} 
are a non-decreasing and diverging sequence $\{\gamma_{\alpha,k,n}\}_{n \in \mb{N}\setminus \{0\}}$ (we repeat each eigenvalue according to its multiplicity).
Let,  for future reference,  
\begin{align}
	&V_{\alpha,k,n} \text{ be the eigenspace of problem \eqref{prob_eigenvalue_sphere} associated to the eigenvalue } \gamma_{\alpha,k,n},\label{def_eigenspace}\\
	&M_{\alpha,k,n} \text{ be the dimension of } V_{\alpha,k,n}, \label{def_dimension-eigenspace}	\\
	&\{Z_{\alpha,k,n,i}: i \in \{1, \dots,M_{\alpha,k,n}\}\} \text{ be a $L^2(\mb{S^+},\theta_{N+1}^{1-2s})$ orthonormal basis of } V_{\alpha,k,n} \label{def_basis_eigenfunction}\\
	&\text{ of eigenfunctions of problem  \eqref{prob_eigenvalue_sphere}}. 
\end{align}
Finally  $\{Z_{\alpha,k,n}\}_{n \in \mathbb{N}\setminus\{0\}}:=\bigcup_{n=1}^\infty\{Z_{\alpha,k,n,i}: i \in \{1, \dots,M_{\alpha,k,n}\}\}$ is an orthonormal basis of $L^2(\mb{S^+},\theta_{N+1}^{1-2s})$.

\begin{remark}\label{reamrk_eigenfunction0}
It is worth noticing that  $Z_{\alpha,k,n}$ cannot vanish identically  on $\mb{S}'$. We argue by contradiction. In view of  \cite[Lemma 2.1]{FF}, we can show  with a direct computation that 
$V(z):=|z|^{-\frac{N-2s}{2}+\sqrt{\left(\frac{N-2s}{2}\right)^2+\gamma_{\alpha,k,n}}}Z_{\alpha,k,n}(z/|z|)$ solves $\dive(y^{1-2s}\nabla V)-y^{1-2s}\frac{\alpha}{|x|_k^2}V=0$ on $\R^{N+1}_+$ and satisfies both zero
Dirichlet and zero Neumann condition  on $\R^N \times \{0\}$. Let 
\begin{equation}\label{def_Sigma}
	\Sigma_k:=\{z \in \R^{N+1}:|x|_k=0\}.
\end{equation}
Note that  $\Sigma_k$ has Lebesgue measure $0$ and that $V$ is a solution to an elliptic equitation with  a Muckenhoupt weight  and  bounded coefficients  away from $\Sigma_k$. Then by the unique 
continuation principles proved in  \cite{TZ},   we conclude that $V \equiv 0$. Hence $Z_{\alpha,k,n}\equiv0$ which is a contradiction.
\end{remark}

\subsection{Inequalities in $H^1(B_r^+,y^{1-2s})$}
In this subsection we prove some useful inequalities.
\begin{proposition}\label{prop_hardy_with_boundary}
For any $r >0$, any $k \in \{0,\dots,N\}$, and any $V \in H^1(B_r^+,y^{1-2s})$
\begin{equation}\label{ineq_hardy_with_boundary}
\left(\frac{k-2}{2}\right)^2\int_{B_r^+}y^{1-2s}\frac{V^2}{|x|_k^2}\, dz \le \int_{B_r^+}y^{1-2s}|\nabla V|^2\, dz +\frac{N-2s}{2r}\int_{S_r^+}y^{1-2s}V^2\, dz.
\end{equation}
\end{proposition}

\begin{proof}
By density it is enough to prove \eqref{ineq_hardy_with_boundary} for any $\phi \in C^{\infty}(\overline{B_r^+})$.  
Passing in polar coordinates, by \eqref{ineq_hardy_sphere} and \cite[Lemma 2.4]{FF}, we have that 
\begin{multline}\label{dim_hardy_with_boundary:1}
\left(\frac{k-2}{2}\right)^2\int_{B_r^+}y^{1-2s}\frac{V^2}{|x|_k^2}\, dz
= \left(\frac{k-2}{2}\right)^2\int_{0}^r \rho^{N-1-2s} \left(\int_{\mb{S}^+}\frac{V^2(\rho\theta )}{|\theta|^2_k}\, dS\right)\, d\rho\\
\le\int_{0}^r \rho^{N-1-2s} \left(\left(\frac{N-2s}{2}\right)^2\int_{\mb{S}^+}\theta_{N+1}^{1-2s}|V^2(\rho\theta )|^2 \, dS
+\int_{\mb{S}^+}\theta_{N+1}^{1-2s}|\nabla_{\mb{S}} V(\rho\theta )|^2 \, dS\right)\, d\rho \\
=\left(\frac{N-2s}{2}\right)^2\int_{B_r^+}y^{1-2s}\frac{V^2}{|z|^2}\, dz +\int_{0}^r \rho^{N-1-2s} \left(\int_{\mb{S}^+}\theta_{N+1}^{1-2s}|\nabla_{\mb{S}} V(\rho\theta )|^2 \, dS\right)\, d\rho\\
\le \frac{N-2s}{2r}\int_{S_r^+}y^{1-2s}V^2\, dS \\
+\int_{0}^r \rho^{N+1-2s} \left(\int_{\mb{S}^+}\theta_{N+1}^{1-2s}\left(\frac{1}{\rho^2}|\nabla_{\mb{S}} V(\rho\theta )|^2+\left|\pd{V}{\rho}(\rho \theta)\right|^2\right) \, dS\right)\, d\rho \\
=\frac{N-2s}{2r}\int_{S_r^+}y^{1-2s}V^2\, dS + \int_{B_r^+}y^{1-2s} |\nabla V|^2 \, dz,
\end{multline}
hence we have proved \eqref{ineq_hardy_with_boundary}. 
\end{proof}

\begin{proposition}\label{prop_h_D+H}
Let $r>0$ and suppose that $h:B_r':\to \R$ is a measurable function such that  
\begin{equation}\label{hp_h}
|h(x)| \le C_h |x|^{-2s+\e} \quad \text{ for a.e. } x \in B_r',
\end{equation}
for some positive constant $C_h$ and some $\e \in (0,1)$.
Then for any $k \in \{3,\dots,N\}$, any $\alpha$ as in  \eqref{hp_alpha} and  any  $V \in H^1(B_r^+,y^{1-2s})$ 
\begin{multline}\label{ineq_h_D+H}
\int_{B_r'}|h| \Tr(V)^2 \, dx \\
\le k_{N,s,h}r^{\e} \left(\int_{B_r^+}y^{1-2s}|\nabla V|^2\, dz-\int_{B_r^+}y^{1-2s}\frac{\alpha}{|x|_k^2}V^2\, dz +\frac{N-2s}{2r}\int_{S_r^+}y^{1-2s}V^2\, dz\right),
\end{multline}
where $ k_{N,s,h}$ is a positive constant depending only on $N,s,C_h$.
\end{proposition}

\begin{proof}
The claim follows from \eqref{hp_h}, \cite[Lemma 2.5]{FF},  and \eqref{ineq_hardy_with_boundary}.
\end{proof}

In view of \eqref{hp_alpha} there exists  $r_0>0$ such that 
\begin{equation}\label{def_r0}
\overline{B_{r_0}^+} \subset C \quad \text{ and  } \quad \alpha\left(\frac{2}{k-2}\right)^2+ c_{N,s}k_{N,s,g} r_0^\e<1,
\end{equation} 
where $k_{N,s,g}$ is as in Proposition \ref{prop_h_D+H}, $c_{N,s}$ as in Theorem \ref{theorem_extension} and $g$ as in \eqref{hp_g}.
\begin{proposition}\label{prop_nabla_D+H}
Let $k \in \{3,\dots,N\}$, $\alpha$ as \eqref{hp_alpha}, $g$ as in \eqref{hp_g}, $c_{N,s}$ as in Theorem \ref{theorem_extension} and $r_0$ as in \eqref{def_r0}. Then for any $V \in H^1(B_{r}^+,y^{1-2s})$ and any $r \in (0,r_0]$
\begin{multline}\label{ineq_nabla_D+H}
\int_{B_{r}^+} y^{1-2s}\left(|\nabla W|^2-\frac{\alpha}{|x|_k^2}W^2  \right) \, dz\\
-c_{N,s}\int_{B_{r}'} g \Tr(W)^2 \, dx +\frac{N-2s}{2r}\int_{S_{r}^+} y^{1-2s} W^2dS \\
\ge\left(1-\alpha\left(\frac{2}{k-2}\right)^2+ c_{N,s}k_{N,s,g} r_0^\e\right)\left(\int_{B_{r}^+} y^{1-2s} |\nabla W|^2 \, dz+\frac{N-2s}{2r}\int_{S_{r}^+} y^{1-2s} W^2dS\right).
\end{multline}
\end{proposition}
\begin{proof}
The claim follows from  Proposition \ref{prop_h_D+H},   \eqref{hp_g} and \eqref{ineq_hardy_with_boundary}.
\end{proof}

\section{Approximated problems and a Pohozaev-type Identity}\label{sec_Pohozaev_identity}
In order to obtain a Pohozaev type identity for a  weak solution of \eqref{prob_extension_with_equation}, we  approximate it with a family of solutions to more regular problems. Then we obtain a  Pohozaev-type identity for such solutions and  pass to the limit.

Let for any $r>0$
\begin{equation}\label{def_H10Br}
H^{1}_{0,S_r^+}(B_r^+,y^{1-2s}):=\overline{\{\phi \in C^{\infty}(\overline{B_r^+}):\phi=0 \text{ on } S_r^+\}}^{\norm{\cdot}_{H^1(B_r^+,y^{1-2s})}}.
\end{equation}

\begin{remark}\label{remark_equivalence_norm_H10Sr+}
 Let $r_0$ be as in \eqref{def_r0}. By \eqref{ineq_nabla_D+H} and the Poincaré  inequality, for any $r \in (0,r_0)$,
\begin{equation}\label{def_norm_galphak}
\norm{W}_{g,\alpha,k,0}:=\left(\int_{B_{r}^+} y^{1-2s}\left(|\nabla W|^2-\frac{\alpha}{|x|_k^2}W^2  \right) \, dz-c_{N,s}\int_{B_{r}'} g \Tr(W)^2 \, dx \right)^\frac12
\end{equation}
defines a norm on $H^{1}_{0,S_{r}^+}(B_{r}^+,y^{1-2s})$ equivalent to   \eqref{def_norm_H1t}.
Furthermore 
\begin{equation}\label{def_norm_galphak_withboundray}
\norm{W}_{g,\alpha,k}
:=\left(\int_{B_{r}^+} y^{1-2s}\left(|\nabla W|^2-\frac{\alpha}{|x|_k^2}W^2  \right) \, dz-c_{N,s}\int_{B_{r}'} g \Tr(W)^2 \, dx
 +\int_{S_{r}^+} y^{1-2s} W^2 \, dz\right)^\frac12
\end{equation}
defines a norm on $H^{1}(B_{r}^+,y^{1-2s})$ equivalent to  \eqref{def_norm_H1t}.
\end{remark}

\begin{theorem}\label{theorem_aproximaition}
Let $U$ be a weak solutions of \eqref{prob_extension_with_equation},  and $r_0$ as in \eqref{def_r0}.  Then there exists $\tilde{\la}>0$ such that for any $\la \in (0,\tilde{\la})$ the problem 
\begin{equation}\label{prob_approximated}
\begin{cases}
-\dive(y^{1-2s}\nabla V)= y^{1-2s} \frac{\alpha}{|x|_k^2+\la^2}V, \quad &\text{ in } B_{r_0}^+,\\
V=U, \quad &\text{ on } S_{r_0}^+,\\
-\lim_{y \to 0^+}y^{1-2s}\pd{V}{y}=c_{N,s}g\Tr(V), \quad &\text{ on } B_{r_0}',
\end{cases}
\end{equation}
where $c_{N,s}>0$ is as in Theorem \ref{theorem_extension}, admits a weak solution $U_\la \in H^1(B_{r_0}^+,y^{1-2s})$, i.e.
\begin{equation}\label{eq_approximated}
\int_{B_{r_0}^+} y^{1-2s} \nabla U_\la\cdot\nabla W\, dz-\int_{B_{r_0}^+} y^{1-2s}\frac{\alpha}{|x|_k^2+\la^2}U_\la W  \, dz
=c_{N,s}\int_{B_{r_0}'} g \Tr(V)\Tr(W) \, dx 
\end{equation}
for any  $W \in H^1_{0,S_{r_0}^+}(B_{r_0}^+,y^{1-2s})$, and $U_\la=U$ on $S_{r_0}^+$. Furthermore 
\begin{equation}\label{limit_Ula_to_U}
U_\la \to U \text{ strongly in } H^1(B_{r_0}^+,y^{1-2s}) \quad \text{ as } \la \to 0^+.
\end{equation}
\end{theorem} 
\begin{proof}
Let us consider the map $\Phi:\R\times  H^1_{0,S_r^+}(B_r^+,y^{1-2s}) \to  (H^1_{0,S_r^+}(B_r^+,y^{1-2s}))^*$ defined as 
\begin{multline}\label{dim_aproximaition:1}
\Phi(\la,V)(W):=\int_{B_{r_0}^+} y^{1-2s} \nabla V\cdot\nabla W \, dz-\int_{B_{r_0}^+} y^{1-2s}\frac{\alpha}{|x|_k^2+\la^2}V W  \, dz\\
-c_{N,s}\int_{B_{r_0}'} g \Tr(V)\Tr(W) \, dx+\int_{B_{r_0}^+} y^{1-2s}\left(\frac{\alpha}{|x|_k^2+\la^2}-\frac{\alpha}{|x|^2_k}\right) U W \, dz.
\end{multline}
for any $W \in H^1_{0,S_{r_0}^+}(B_{r_0}^+,y^{1-2s})$. It is clear that $\Phi$ is well defined and that $\Phi$ is continuous in $(0,0)$ in view of  H\"older's inequality,  Proposition \ref{prop_h_D+H},   \eqref{hp_g},  and \eqref{ineq_hardy_with_boundary}.  Furthermore  $\Phi(0,0)=0$. 

Let us prove that $\Phi_V(0,0) \in \mathcal{L}(H^1_{0,S_{r_0}^+}(B_{r_0}^+,y^{1-2s}),(H^1_{0,S_{r_0}^+}(B_{r_0}^+,y^{1-2s})^*)$ is an isomorphism, where $\Phi_V$ is the partial derivative with respect to $V$ of $\Phi$.
For any  $W_1,W_2 \in H^1_{0,S_{r_0}^+}(B_{r_0}^+,y^{1-2s})$
\begin{equation}\label{dim_aproximaition:2}
\df{(H^1_{0,S_{r_0}^+}(B_{r_0}^+,y^{1-2s}))^*}{\Phi_V(0,0)(W_1)}{W_2}{H^1_{0,S_{r_0}^+}(B_{r_0}^+,y^{1-2s})}=\ps{g,\alpha,k,0}{W_1}{W_2}.
\end{equation} 
Hence, by Remark \ref{remark_equivalence_norm_H10Sr+}, $\Phi_V(0,0)$ is the Rietz isomorphism associated to the norm $\norm{\cdot}_{g,\alpha,k,0}$.

We are now in position to apply the Implicit Function Theorem to  $\Phi$ in the point $(0,0)$ and conclude that there exist $\tilde{\la} >0$, $\rho>0$   and a  function
\begin{equation}\label{dim_aproximaition:4}
f:(-\tilde\la,\tilde\la) \to B_\rho(0),
\end{equation}
continuous in $0$, such that   $\Phi(\la,V)=0$ if and only if $V=f(\la)$ for any $\la \in 
(-\tilde\la,\tilde{\la})$ and $V \in  B_\rho(0)$.
The set  $B_\rho(0)$ in \eqref{dim_aproximaition:4} is defined as $B_\rho(0)=\{V \in H^1_{0,S_{r_0}^+}(B_{r_0}^+,y^{1-2s}):\norm{V}_{H^1(B_{r_0}^+,y^{1-2s})}<\rho\}$.

It follows that $U_\la:=U-f(\la)$ solves \eqref{eq_approximated} for any $\la \in (0,\tilde \la)$ since $U$ is a solution of \eqref{eq_extension_with_boundary}. Furthermore $U_\la \to U$ strongly in  $H^1(B_{r_0}^+,y^{1-2s})$ as $\la \to 0^+$ since $f$ is continuous in $0$ and $f(0)=0$.
\end{proof}

\begin{remark}\label{reamrk_extension_with_boundary_double}
Let $U_\la$ be a solution of \eqref{eq_approximated}. Then,  reasoning in the same way of Proposition \ref{prop_extension_with_boundary}, we can prove that for a.e. $r \in (0,r_0)$, a.e. $\rho \in 
(0,r)$  and any $W \in H^1(B_r^+\setminus {B^+_\rho},y^{1-2s})$
\begin{multline}\label{eq_extension_with_boundary_double}
\int_{B_r^+\setminus B_\rho^+} y^{1-2s} \left(\nabla U_\la\cdot\nabla W -\frac{\alpha}{|x|_k^2+\la^2}U_\la W  \right)\, dz	\\
= \frac{1}{r}\int_{S_r^+} y^{1-2s} \nabla U_\la\cdot z  \,W \, dS-\frac{1}{\rho}\int_{S_\rho^+} y^{1-2s} \nabla U_\la\cdot z  \,W \, dS
+c_{N,s}\int_{B_r'\setminus B'_\rho} g\Tr(U_\la)  \Tr(W) \, dx.
\end{multline}
\end{remark}

Let $\nu$ be the outer normal vector to $B_r^+$ on $S_r^+$, that is $\nu(z) =\frac{z}{|z|}$.
\begin{proposition}\label{prop_pohozaev_identity_Ula}
For any $\la \in (0,\tilde{\la})$, let $U_\la$ be a solution of \eqref{eq_approximated}. Then for a.e. $ r \in (0,r_0)$
\begin{align}\label{eq_pohozaev_identity_Ula}
&\frac{r}{2}\int_{S_r^+}y^{1-2s}|\nabla U_\la|^2 \, dS-r \int_{S_r^+} y^{1-2s} |\nabla U_\la \cdot \nu|^2 \, dS \\
+&\frac{c_{N,s}}{2}\int_{B_r'} (Ng + x\cdot\nabla g) |\Tr(U_\la)|^2 \, dx -\frac{c_{N,s}r}{2}\int_{S_r'} g |\Tr(U_\la)|^2 \, dS \\
=&\frac{N-2s}{2}\int_{B_r^+}y^{1-2s} |\nabla U_\la |^2 \, dz +  \int_{B_r^+}y^{1-2s} \frac{\alpha}{|x|^2_k+\la^2}U_\la\nabla  U_\la \cdot z \, dz.
\end{align}
\end{proposition}
\begin{proof}
We proceed in the spirit of \cite[Proposition 2.3]{FSlincei}, since   $(|x|^2_k+\la^2)^{-1}U_\la \in L^2(B_r^+,y^{1-2s})$ and $g \in W^{1,\infty}_{loc}(\Omega \setminus\{0\})$.
Then by \cite[Theorem 2.1, Proposition 3.6]{FSlincei} and the proof of \cite[Proposition 2.2]{FSlincei}, for any $r\in (0,r_0)$ and $\rho \in (0,r)$,
\begin{align}
&\nabla_x U_\la \in H^1(B_r^+\setminus B_{\rho}^+,y^{1-2s}), \quad \text{ and } \quad y^{1-2s}\pd{U_\la}{y} \in   H^1(B_r^+\setminus B_{\rho}^+,y^{2s-1}), \label{dim_pohozaev_identity_Ula:1}\\
&\Tr(U_\la) \in H^{1+s}(B_r'\setminus B_{\rho}'), \quad  \text{ and } \quad \Tr(\nabla_x U_\la)=\nabla \Tr(U_\la),\label{dim_pohozaev_identity_Ula:1.1}\\
& \nabla U_\la \cdot z \in H^1(B_r^+\setminus B_{\rho}^+,y^{1-2s}), \quad \text{ and } \quad \Tr(\nabla U_\la \cdot z)=\Tr(\nabla U_\la) \cdot x,
\end{align} 
where  $H^{1+s}(B_r'\setminus B_{\rho}'):=\{w \in H^1(B_r'\setminus B_{\rho}'): \pd{w}{x_i} \in  W^{s,2}(B_r'\setminus B_{\rho}') \text{ for any } i=1,\dots, N\}$.
We also have, in view of \eqref{eq_approximated},  the following  identity 
\begin{equation}\label{dim_pohozaev_identity_Ula:2}
\dive(y^{1-2s}|\nabla U_\la|^2 z-2y^{1-2s}\nabla U_\la \cdot z \nabla U_\la)=(N-2s)|\nabla U_\la|^2 +2\frac{\alpha}{|x|_k^2+\la^2}U_\la \nabla U_\la \cdot z
\end{equation}
in a distributional sense in $B_r^+\setminus B_{\rho}^+$.
Furthermore, thanks to  \eqref{dim_pohozaev_identity_Ula:1},
\begin{equation}\label{dim_pohozaev_identity_Ula:3}
\dive(y^{1-2s}\nabla U_\la \cdot z \nabla U_\la)= -y^{1-2s}\frac{\alpha}{|x|_k^2+\la^2}U_\la \nabla U_\la \cdot z+ y^{1-2s}\nabla U_\la \cdot \nabla(\nabla U_\la\cdot z) \in L^1(B_r^+\setminus B_{\rho}^+)
\end{equation}
and so by \eqref{dim_pohozaev_identity_Ula:2}
\begin{equation}\label{dim_pohozaev_identity_Ula:4}
\dive(y^{1-2s}|\nabla U_\la|^2 z) \in L^1(B_r^+\setminus B_{\rho}^+).
\end{equation}
Let, for any $\delta \in (0,r)$,
\begin{equation}\label{def_BSrdelta}
 B^+_{r,\delta}:=\{(x,y)\ \in B_r^+:y>\delta\} \quad \text { and } \quad S^+_{r,\delta}:=\{(x,y)\ \in S_r^+:y>\delta\}.
\end{equation}
Integrating by part on $B_r^+\setminus B_{\rho}^+$ we obtain, for any $\delta\in (0,\rho)$,  
\begin{multline}\label{dim_pohozaev_identity_Ula:5}
\int_{B_{r,\delta}^+\setminus B_{\rho,\delta}^+}\dive(y^{1-2s}|\nabla U_\la|^2 z) \, dz = r \int_{S_{r,\delta}^+}y^{1-2s}|\nabla U_\la|^2 \, dS
-\rho \int_{S_{\rho,\delta}^+}y^{1-2s}|\nabla U_\la|^2 \, dS \\
-\delta^{2-2s}\int_{B'_{\sqrt{r^2-\delta^2}}\setminus B'_{\sqrt{\rho^2-\delta^2}}}|\nabla U_\la|^2(x,\delta)\, dx.
\end{multline}
We claim that there exists a sequence $\delta_n \to 0^+$ such that 
\begin{equation}\label{dim_pohozaev_identity_Ula:6}
\lim_{n \to \infty} \delta^{2-2s}\int_{B'_{\sqrt{r^2-\delta_n^2}}\setminus B'_{\sqrt{\rho^2-\delta_n^2}}}|\nabla U_\la|^2(x,\delta)\, dx=0
\end{equation}
arguing by contradiction. If the claim does not hold than there exist a constant $C>0$ and $\delta_0 \in (0,\rho)$ such that $B_r'\times(0,\delta_0)\subseteq B_{r_0}^+$ and 
\begin{equation}\label{dim_pohozaev_identity_Ula:7}
\delta^{1-2s}\int_{B'_{\sqrt{r^2-\delta^2}}\setminus B'_{\sqrt{\rho^2-\delta^2}}}|\nabla U_\la|^2(x,\delta)\, dx \ge \frac{C}{\delta} \quad \text{ for any } \delta \in (0,\delta_0).
\end{equation}
Then integrating \eqref{dim_pohozaev_identity_Ula:7} over $(0,\delta_0)$ we obtain 
\begin{equation}\label{dim_pohozaev_identity_Ula:8}
\int_{0}^{\delta_0}\left(\delta^{1-2s}\int_{B'_r}|\nabla U_\la|^2(x,\delta)\, dx\right) \, d \delta 
\ge \int_{0}^{\delta_0} \frac{C}{\delta} d\delta = +\infty,
\end{equation}
which is a contradiction in view of the Fubini-Tonelli Theorem. Then we can pass to the limit as $\delta=\delta_n$ in \eqref{dim_pohozaev_identity_Ula:5} and conclude that, thanks to the Dominate 
Convergence Theorem and the Monotone Convergence Theorem,
\begin{equation}\label{dim_pohozaev_identity_Ula:9}
\int_{B_{r}^+\setminus B_{\rho}^+}\dive(y^{1-2s}|\nabla U_\la|^2 z) \, dz = r \int_{S_{r}^+}y^{1-2s}|\nabla U_\la|^2 \, dS
-\rho \int_{S_{\rho}^+}y^{1-2s}|\nabla U_\la|^2 \, dS 
\end{equation}
for a.e $r\in (0,r_0)$ and a.e. $\rho \in (0,r)$. Testing \eqref{eq_extension_with_boundary_double} with $\nabla U \cdot z$ we obtain, in view of \eqref{dim_pohozaev_identity_Ula:3} and Remark 
\ref{reamrk_extension_with_boundary_double},
\begin{multline}\label{dim_pohozaev_identity_Ula:10}
\int_{B^+_r\setminus B_\rho^+}\dive(y^{1-2s}\nabla U_\la \cdot z \nabla U_\la) \, dz\\
=\int_{B_r^+\setminus B_\rho^+} y^{1-2s}\nabla U_\la \cdot \nabla(\nabla U_\la\cdot z)\, dz-\int_{B_r^+\setminus B_\rho^+} y^{1-2s} \frac{\alpha}{|x|_k^2+\la^2}U_\la \nabla U_\la \cdot z\, dz \\
=\frac{1}{r}\int_{S_r^+} y^{1-2s}|\nabla U_\la \cdot z|^2\, dS-\frac{1}{\rho}\int_{S_\rho^+} y^{1-2s}|\nabla U_\la \cdot z|^2\, dS+ c_{N,s}\int_{B_r'\setminus B_\rho'}g \Tr(U_\la)\, \nabla_x\Tr(U_\la)\cdot x \, dx.
\end{multline} 
We note that $g \Tr(U_\la)^2 x  \in W^{1,1}(B_r'\setminus B_\rho',\R^N)$ by \eqref{hp_g} and \eqref{dim_pohozaev_identity_Ula:1.1} hence integrating by part we obtain 
\begin{multline}\label{dim_pohozaev_identity_Ula:11}
\int_{B_r'\setminus B_\rho'}g\, \Tr(U_\la) \nabla_x\Tr(U_\la)\cdot x \, dx=-\frac{1}{2}\int_{B_r'\setminus B_\rho'}(Ng+x \cdot \nabla g) \Tr(U_\la)^2 \, dx\\
+\frac{r}{2}\int_{S_r'}g|\Tr(U_\la)|^2 dS'-\frac{\rho}{2}\int_{S_\rho'}g|\Tr(U_\la)|^2 dS'.
\end{multline}

Arguing as in the proof of \eqref{dim_pohozaev_identity_Ula:6}, we see that there exists a sequence $\rho_n \to 0^+$ such that 
\begin{multline}\label{dim_pohozaev_identity_Ula:12}
\lim_{n \to \infty}\rho_n \int_{S_{\rho_n}^+}y^{1-2s}|\nabla U_\la|^2 \, dS=\lim_{n \to \infty}\rho_n \int_{S_{\rho_n}^+} y^{1-2s}\left|\nabla U_\la \cdot \frac{z}{|z|}\right|^2\, dS\\
=\lim_{n \to \infty}\rho_n\int_{S_{\rho_n}'}g|\Tr(U_\la)|^2 dS'=0.
\end{multline}
Then by the Dominated Convergence Theorem, we can pass to the limit as $\rho=\rho_n$ and $n \to \infty$ in \eqref{dim_pohozaev_identity_Ula:9}, \eqref{dim_pohozaev_identity_Ula:10}, \eqref{dim_pohozaev_identity_Ula:11} and conclude that \eqref{eq_pohozaev_identity_Ula} holds in view of \eqref{dim_pohozaev_identity_Ula:2}.
\end{proof}

\begin{proposition}\label{prop_pohozaev_ideentity}
Let $U$ be a solution of \eqref{eq_extension_with_equation}. Then for a.e. $ r \in (0,r_0)$
\begin{align}\label{eq_pohozaev_ideentity}
&\frac{r}{2}\int_{S_r^+}y^{1-2s}\left(|\nabla U|^2- \frac{\alpha}{|x|^2_k}U^2\right) \, dS-r \int_{S_r^+} y^{1-2s} |\nabla U \cdot \nu|^2 \, dS \\
+&\frac{c_{N,s}}{2}\int_{B_r'} (Ng +x\cdot\nabla g) |\Tr(U)|^2 \, dx -\frac{c_{N,s}}{2}r\int_{S_r'} g |\Tr(U)|^2 \, dS' \\
=& \frac{N-2s}{2}\int_{B_r^+}y^{1-2s} \left(|\nabla U |^2- \frac{\alpha}{|x|^2_k}U^2\right) \, dz.
\end{align}
\end{proposition}
\begin{proof}
Let $r\in (0,r_0)$ and    $B^+_{r,\delta}$, $S^+_{r,\delta}$ be as in \eqref{def_BSrdelta} for any $\delta \in (0,r)$.
Then, by \eqref{hp_alpha}, 
\begin{multline}\label{dim_pohozaev_identity:1}
\dive\left(y^{1-2s} \frac{\alpha}{|x|^2_k+\la^2} U_\la^2 \, z\right)\\
= y^{1-2s}\left(2\frac{\alpha}{|x|^2_k+\la^2} U_\la \nabla U_\la \cdot z+(N+2-2s) \frac{\alpha}{|x|^2_k+\la^2} U_\la^2-2\frac{\alpha |x|^2_k}{(|x|^2_k+\la^2)^2}U_\la^2\right)
\end{multline}
and $y^{1-2s} \frac{\alpha}{|x|^2_k+\la^2} U_\la^2 z \in W^{1,1}(B^+_{r,\delta},\R^{N+1})$.
Integrating \eqref{dim_pohozaev_identity:1} by part in $B_{r,\delta}^+$ we obtain 
\begin{multline}\label{dim_pohozaev_identity:2}
r\int_{S_{r,\delta}^+}y^{1-2s}\frac{\alpha}{|x|^2_k+\la^2} U_\la^2 \, dS- \delta^{2-2s} \int_{B'_{\sqrt{r^2-\delta^2}}} \frac{\alpha}{|x|^2_k+\la^2} U_\la^2(x,\delta) \, dx\\
= \int_{B_{r,\delta}^+}y^{1-2s}\left(2\frac{\alpha}{|x|^2_k+\la^2} U_\la \nabla U_\la \cdot z+(N+2-2s) \frac{\alpha}{|x|^2_k+\la^2} U_\la^2-2\frac{\alpha |x|^2_k}{(|x|^2_k+\la^2)^2}U_\la^2\right) \, dz.
\end{multline}
We claim that there exists a sequence $\delta_n \to 0^+$  as $n \to \infty$ such that 
\begin{equation}\label{dim_pohozaev_identity:3}
\lim_{n \to \infty}\delta_n^{2-2s} \int_{B'_{\sqrt{r^2-\delta_n^2}}} \frac{\alpha}{|x|^2_k+\la^2} U_\la^2(x,\delta_n) \, dx=0
\end{equation}
arguing by contradiction. If \eqref{dim_pohozaev_identity:3} does not hold, then there exists a constant $C >0$  and $\delta_0 \in (0,r)$ such that $(0,\delta_0) \times B_r' \subseteq B_{r_0}^+$ and  
\begin{equation}\label{dim_pohozaev_identity:4}
\delta^{1-2s} \int_{B'_{\sqrt{r^2-\delta^2}}} \frac{\alpha}{|x|^2_k+\la^2} U_\la^2(x,\delta) \, dx \ge \frac{C}{\delta}
\end{equation}
for any $\delta \in (0,\delta_0)$.
Integrating over $(0,\delta_0)$ we obtain 
\begin{equation}\label{dim_pohozaev_identity:5}
+\infty >\int_{0}^{\delta_0}\delta^{1-2s} \left(\int_{B'_r} \frac{\alpha}{|x|^2_k+\la^2} U_\la^2(x,\delta) \, dx\right) d\delta \ge \int_{0}^{\delta_0}\frac{C}{\delta} \, d \delta,
\end{equation} 
a contradiction in view of the Fubini-Tonelli Theorem. Passing to the limit for  $\delta =\delta_n$ as $n \to \infty$ in \eqref{dim_pohozaev_identity:2} we conclude that 
\begin{multline}\label{dim_pohozaev_identity:6}
\int_{B_r^+}y^{1-2s}\frac{\alpha}{|x|^2_k+\la^2} U_\la \nabla U_\la \cdot z \, dz =\frac{r}{2}\int_{S_r^+}y^{1-2s}\frac{\alpha}{|x|^2_k+\la^2} U_\la^2 \, dS\\
- \frac{1}{2}\int_{B_r^+}y^{1-2s}\left((N+2-2s)\frac{\alpha}{|x|^2_k+\la^2} U_\la^2-2\frac{\alpha |x|^2_k}{(|x|^2_k+\la^2)^2}U_\la^2\right) \, dz.
\end{multline}

Now we pass to the limit as $\la \to 0^+$, eventually along a  suitable sequence $\la_n \to 0^+$, in each term of \eqref{eq_pohozaev_identity_Ula} taking into account    \eqref{dim_pohozaev_identity:6}.
We recall that, by Theorem \ref{theorem_aproximaition}, $U_\la \to U$ strongly in $H^1(B_r^+,y^{1-2s})$ for any $r \in (0,r_0]$.
It is clear that for any $ r \in (0,r_0)$
\begin{equation}\label{dim_pohozaev_identity:6.1}
\lim_{\la \to 0^+}\int_{B_r^+}y^{1-2s} |\nabla U_\la |^2 \, dz=\int_{B_r^+}y^{1-2s} |\nabla U |^2 \, dz.
\end{equation}
Furthermore there exists a sequence $\la_n \to 0$ as $n \to \infty$  and  $G \in L^2(B_{r_0}^+,y^{1-2s}|x|_k^{-2})$ such that 
\begin{align}
&(N+2-2s)\frac{\alpha}{|x|^2_k+\la_n^2} U_{\la_n}^2-2\frac{\alpha |x|^2_k}{(|x|^2_k+\la_n^2)^2}U_{\la_n}^2 
\to (N-2s)\frac{\alpha}{|x|^2_k} U^2 \quad \text{ for a.e. } z \in B_{r_0}^+,\label{dim_pohozaev_identity:6.2}\\
&\frac{\alpha}{|x|^2_k+\la_n^2}U_{\la_n}-\frac{\alpha}{|x|^2_k}U \to 0 \quad \text{ for a.e. } z \in B_{r_0}^+\label{dim_pohozaev_identity:6.2.1},\\
& |U_{\la_n}|\le |G| \quad \text{ for a.e. } z \in B_{r_0}^+ \text{ and any } n \in \mathbb{N}.
\end{align}
Then by the Dominated Convergence Theorem we conclude that for any $r \in (0,r_0)$ 
\begin{multline}\label{dim_pohozaev_identity:6.3}
\lim_{n \to \infty}\int_{B_r^+}y^{1-2s}\left((N+2-2s)\frac{\alpha}{|x|^2_k+\la_n^2} U_{\la_n}^2-2\frac{\alpha |x|^2_k}{(|x|^2_k+\la_n^2)^2}U_{\la_n}^2\right) \, dz\\
=(N-2s)\int_{B_r^+}y^{1-2s}\frac{\alpha}{|x|^2_k} U^2\, dz
\end{multline}
and 
\begin{equation}\label{dim_pohozaev_identity:6.3.1}
\lim_{n \to \infty}\int_{B_r^+}y^{1-2s}\left|\frac{\alpha}{|x|^2_k+\la_n^2}U_{\la_n}^2-\frac{\alpha}{|x|^2_k}U^2\right| \, dz=0.
\end{equation}
By \eqref{hp_g}, \eqref{ineq_h_D+H}, \eqref{ineq_hardy_with_boundary} and Proposition \ref{prop_trace_Sr+}
\begin{equation}\label{dim_pohozaev_identity:6.4}
\lim_{\la \to 0^+}\int_{B_r'} |Ng +\nabla g\cdot x|\, |\Tr(U_\la)-\Tr(U)|^2 \, dx=0
\end{equation}
hence, for any $r \in (0,r_0)$,
\begin{equation}\label{dim_pohozaev_identity:6.5}
\lim_{\la \to 0^+}\int_{B_r'} (Ng +x\cdot \nabla g) |\Tr(U_\la)|^2 \, dx= \int_{B_r'} (Ng +\nabla g\cdot x)|\Tr(U)|^2\, dx.
\end{equation}
By Fatou's Lemma and the Coarea Formula,
\begin{equation}\label{dim_pohozaev_identity:7}
\int_{0}^{r_0}\left(\liminf_{\la \to 0^+}\int_{S_r^+}y^{1-2s}|\nabla U_\la -\nabla U|^2 \, dS \right) dr \le \liminf_{\la \to 0^+}\int_{B_{r_0}^+}y^{1-2s}|\nabla U_\la -\nabla U|^2 \, dS=0,
\end{equation}
and  so 
\begin{equation}\label{dim_pohozaev_identity:8}
\liminf_{\la \to 0^+}\int_{S_{r}^+}y^{1-2s}|\nabla U_\la|^2 \, dS= \int_{S_{r}^+}y^{1-2s}|\nabla U|^2 \, dS
\end{equation}
 for a.e. $r \in (0,r_0)$. Similarly, for a.e. $r \in (0,r_0)$
\begin{equation}\label{dim_pohozaev_identity:9}
\liminf_{\la \to 0^+}\int_{S_{r}^+}y^{1-2s}|\nabla U_\la\cdot \nu|^2 \, dS= \int_{S_{r}^+}y^{1-2s}|\nabla U\cdot \nu|^2 \, dS,
\end{equation}
and, by \eqref{dim_pohozaev_identity:6.4} and  Fatou's Lemma, 
\begin{equation}\label{dim_pohozaev_identity:10}
\liminf_{\la \to 0^+}\int_{S'_r}g |\Tr(U_\la)|^2 \, d'S= \int_{S'_r}g |\Tr(U)|^2\, dS'.
\end{equation}
Furthermore passing to the limit for $\la=\la_n$ as $n \to \infty$ and $\la_n$ is as  in  \eqref{dim_pohozaev_identity:6.2.1}, we obtain
\begin{equation}\label{dim_pohozaev_identity:11}
\lim_{n  \to \infty}\int_{S_{r}^+}y^{1-2s}\frac{\alpha}{|x|^2_k+\la_n^2}U_{\la_n}^2 \,dS=\int_{S_{r}^+}y^{1-2s}\frac{\alpha}{|x|^2_k}U^2 \, dS
\end{equation}
for a.e. $r \in (0,r_0)$, thanks to  Fatou's Lemma and \eqref{dim_pohozaev_identity:6.3.1}.
In conclusion \eqref{eq_pohozaev_ideentity} holds.
\end{proof}

\section{The Monotonicity Formula}\label{sec_monotonicity} 
Let $U$ be a non-trivial solution of \eqref{eq_extension_with_equation}, let $r_0$ be  as in \eqref{def_r0}. For any $r \in (0,r_0]$ we define the height and energy functions respectively as 
\begin{align}
&H(r):=\frac{1}{r^{N+1-2s}}\int_{S_r^+} y^{1-2s}U^2 \, dS,\label{def_H}\\
&D(r):=\frac{1}{r^{N-2s}}\left(\int_{B_r^+} y^{1-2s}\left(|\nabla U|^2-\frac{\alpha}{|x|^2_k}U^2\right) \, dz-c_{N,s}\int_{B_r'}g|\Tr(U)|^2 \, dx\right).\label{def_D}
\end{align}

The proof of the next Proposition is very similar to \cite[Lemma 3.1]{DFV} and we omit it.
We also recall that  $\nu$ is the outer normal vector to $B_r^+$ on $S_r^+$, that is $\nu(z) =\frac{z}{|z|}$.
\begin{proposition}\label{prop_H'}
We have that $H \in W^{1,1}_{loc}((0,r_0])$ and 
\begin{equation}
H'(r)=\frac{2}{r^{N+1-2s}}\int_{S_r^+} y^{1-2s}\pd{U}{\nu}U\, dS=\frac{2}{r}D(r),\label{eq_H'} 
\end{equation}
in a distributional sense and for a.e. $ r \in (0,r_0)$.
\end{proposition}

\begin{proposition} \label{prop_H_not_0}
Let $H$ be as in \eqref{def_H}. Then $H(r)>0$ for any $r \in (0,r_0]$.
\end{proposition}
\begin{proof}
Assume by contradiction that there exists $r \in (0,r_0]$ such that $H(r)=0$. From \eqref{eq_extension_with_boundary} and Remark \ref{remark_equivalence_norm_H10Sr+} we deduce that $U\equiv 0$ on $B_r^+$. Let  $\Sigma_k$ be as in \eqref{def_Sigma}. The function $U$ is a solution of an elliptic equation with bounded coefficients away from $\Sigma_k$ and $\R^N \times\{0\}$. Then the claim follows from    classical unique continuation principles, see for example  \cite{W}.
\end{proof}

\begin{proposition}\label{prop_D'}
The function $D$ defined in \eqref{def_D} belongs to $W^{1,1}_{loc}((0,r_0])$ and 
\begin{equation}\label{eq_D'}
D'(r)=\frac{2}{r^{N+1-2s}}\left( r\int_{S_r^+} y^{1-2s} |\nabla U \cdot \nu|^2 \, dS-c_{N,s}\int_{B_r'} \left(sg +\frac{1}{2}x\cdot\nabla g\right) |\Tr(U)|^2 \, dx \right)
\end{equation}
in a distributional sense and for a.e. $r \in (0,r_0)$.
\end{proposition}
\begin{proof}
By the Coarea Formula 
\begin{align}\label{dim_D':1}
D'(r)=(2s-N)&r^{-N+2s-1}\left(\int_{B_r^+} y^{1-2s}\left(|\nabla U|^2-\frac{\alpha}{|x|^2_k}U^2\right) \, dz-c_{N,s}\int_{B_r'}g|\Tr(U)|^2 \, dx\right)\\
+&r^{-N+2s}\left(\int_{S_r^+} y^{1-2s}\left(|\nabla U|^2-\frac{\alpha}{|x|^2_k}U^2\right) \, dS-c_{N,s}\int_{S_r'}g|\Tr(U)|^2 \, dS'\right)
\end{align}
and so \eqref{eq_D'} follows from \eqref{eq_pohozaev_ideentity}. Furthermore  $D \in W^{1,1}_{loc}((0,r_0])$ by  \eqref{eq_D'}, \eqref{def_D} and  the Coarea Formula.
\end{proof}

Let us define, for any $r \in (0,r_0]$, the frequency function $\mc{N}$ as 
\begin{equation}\label{def_N}
\mc{N}(r):=\frac{D(r)}{H(r)}.
\end{equation}
In view of Proposition \ref{prop_H_not_0} the definition of $\mc{N}$ is well-posed.
\begin{proposition}\label{prop_N'}
We have that $\mc{N} \in W^{1,1}_{loc}((0,r_0])$ and for any $r \in (0,r_0]$
\begin{equation}\label{ineq_N_lower_estimate}
\mathcal{N}(r)>-\frac{N-2s}{2}.
\end{equation}
Furthermore 
\begin{equation}\label{eq_N'}
\mathcal{N}'(r)=v_1(r)+v_2(r)
\end{equation}
in a distributional sense and for a.e. $r \in (0,r_0)$, where 
\begin{equation}\label{def_v1}
v_1(r):=\frac{2r\left(\left(\int_{S_r^+}y^{1-2s}U^2\, dS\right)\left(\int_{S_r^+}y^{1-2s}\left|\pd{U}{\nu}\right|^2\, dS\right)-\left(\int_{S_r^+}y^{1-2s}U \pd{U}{\nu} \, dS\right)^2\right)}
{\left(\int_{S_r^+}y^{1-2s}U^2\, dS\right)^2},
\end{equation}
and 
\begin{equation}\label{def_v2}
v_2(r):=-c_{N,s}\frac{\int_{B_r'} \left(2sg +x\cdot\nabla g\right) |\Tr(U)|^2 \, dx}{\int_{S_r^+}y^{1-2s}U^2 \, dS}.
\end{equation}
Finally 
\begin{equation}\label{ineq_v1_positive}
v_1(r) \ge 0 \quad \text{ for any  } r \in  (0,r_0].
\end{equation}
\end{proposition}
\begin{proof}
Since $1/H,D \in W^{1,1}_{loc}((0,r_0])$ it follows that $\mc{N} \in W^{1,1}_{loc}((0,r_0])$.  We can deduce \eqref{ineq_N_lower_estimate} directly from \eqref{ineq_nabla_D+H} and \eqref{def_N}.
Furthermore by \eqref{eq_H'}
\begin{equation}\label{dim_N'}
\frac{d}{dr}\mc{N}'(r)=\frac{D'(r)H(r)-D(r)H'(r)}{H^2(r)}=\frac{D'(r)H(r)-\frac{r}{2}(H'(r))^2}{H^2(r)}
\end{equation}
and so \eqref{eq_N'} follows from  \eqref{def_H}, \eqref{def_D} and \eqref{eq_D'}. 
Finally \eqref{ineq_v1_positive} is a consequence of the Cauchy-Schwartz inequality in $L^2(S_r^+,y^{1-2s})$ between the vectors $U$ and $\pd{U}{\nu}$.
\end{proof}

\begin{proposition}\label{prop_ineq_v2}
There exists a constant $C>0$ such that 
\begin{equation}\label{ineq_v2}
|v_2(r)| \le Cr^{-1+\e}\left(\mathcal{N}(r)+\frac{N-2s}{2}\right) \quad \text{ for any } r \in (0,r_0].
\end{equation}
\end{proposition}
\begin{proof}
The claim follows from \eqref{hp_g}, \eqref{ineq_h_D+H}, \eqref{ineq_nabla_D+H} and \eqref{def_v2}.
\end{proof}

\begin{proposition}\label{prop_N_upper}
There exists a constant $C_1>0$ such that 
\begin{equation}\label{ineq_N_upper_estimate}
\mc{N}(r)\le C_1 \quad \text{ for any } r \in (0,r_0].
\end{equation}
\end{proposition}
\begin{proof}
Thanks to Proposition \ref{prop_N'}, for a.e. $r \in (0,r_0)$
\begin{equation}\label{dim_N_upper:1}
\left(\mathcal{N}+\frac{N-2s}{2}\right)'(r) \ge v_2(r)\ge -Cr^{-1+\e}\left(\mathcal{N}(r)+\frac{N-2s}{2}\right).
\end{equation}
Hence an integration over $(r,r_0)$ yields
\begin{equation}\label{dim_N_upper:2}
\mathcal{N}(r) \le -\frac{N-2s}{2}+\left(\mc{N}(r_0)+\frac{N-2s}{2}\right)e^{\frac{C}{\e}r_0^\e}
\end{equation}
for any $r \in (0,r_0)$.
\end{proof}

\begin{proposition}\label{prop_limit_N}
The limit 
\begin{equation}\label{limit_N}
\gamma:=\lim_{r \to 0^+} \mc{N}(r)
\end{equation}
exists and it is finite.
\end{proposition}
\begin{proof}
Since $\mc{N} \in W^{1,1}_{loc}((0,r_0])$ by Proposition \ref{prop_N'}, for any $r \in (0,r_0)$
\begin{equation}\label{dim_limit_N:1}
\mc{N}(r)= \mc{N}(r_0) -\int_{r}^{r_0} \mathcal{N}'(r) \, dr =\mc{N}(r_0) -\int_{r}^{r_0}v_1(r) \, dr-\int_{r}^{r_0}v_2(r) \, dr.
\end{equation}
Since $v_1 \ge 0$ by \eqref{ineq_v1_positive} and $v_2 \in L^1(0,r_0)$ by \eqref{ineq_v2} and \eqref{ineq_N_upper_estimate}, we can pass to the limit as $r \to 0^+$ in \eqref{dim_limit_N:1} and conclude that the limit \eqref{limit_N} exists. From \eqref{ineq_N_lower_estimate} and \eqref{ineq_N_upper_estimate} it is finite.
\end{proof}

The proofs of the Proposition \ref{prop_ineq_H} and   \ref{prop_limit_H_exists_finite} are standard and we omit them, 
see for example \cite[Lemma 3.7, Lemma 4.6]{DFV},  \cite[Lemma 5.6, Lemma 6.4]{FFTmagnetic} or \cite[Lemma 5.9, Lemma 6.6]{FFT}.
\begin{proposition}\label{prop_ineq_H}
Let $\gamma$ be as in \eqref{limit_N}.  Then there exists a constant $K>0$ such that 
\begin{equation}\label{ineq_H_upper_estimate}
H(r) \le K r^{2\gamma} \quad \text{ for any } r \in (0,r_0).
\end{equation}
Furthermore for any $\sigma>0$ there exist a constant $K_\sigma$  such that 
\begin{equation}\label{ineq_H_lower_estimate}
H(r) \ge K_\sigma r^{2\gamma +\sigma}\quad \text{ for any } r \in (0,r_0).
\end{equation}
\end{proposition}

\begin{proposition}\label{prop_limit_H_exists_finite}
Let $\gamma$ be as in \eqref{limit_N}. Then there exists the limit 
\begin{equation}\label{limit_H_exists_finite}
\lim_{r \to 0^+}r^{-2\gamma}H(r)
\end{equation}
and it is finite.
\end{proposition}

\section{The Blow-Up Analysis}\label{sec_the_blow_up_analysis}
Let $U$ be a non-trivial solution of \eqref{eq_extension_with_equation} and let $r_0$ be  as in \eqref{def_r0}.   For any $\la \in (0,r_0]$ let 
\begin{equation}\label{def_blow_up_solution}
V^\la(z):=\frac{U(\la z)}{\sqrt{H(\la)}}.
\end{equation}
By a change of variables, it is clear  that  $V^\la$ weakly solves 
\begin{equation}\label{prob_blow_up_solution}
\begin{cases}
-\dive(y^{1-2s}\nabla V^\la)= y^{1-2s} \frac{\alpha}{|x|_k^2}V^\la, \quad &\text{ in } B_{r_0/\la}^+,\\
-\lim_{y \to 0^+}y^{1-2s}\pd{V^\la}{y}=c_{N,s}\la^{2s}g(\la\cdot )\Tr(V^\la), \quad &\text{ on } B'_{r_0/\la},
\end{cases}
\end{equation}
 in the sense that 
\begin{equation}\label{eq_blow_up_solution}
\int_{ B_{r_0/\la}^+} y^{1-2s} \nabla V^\la\cdot\nabla W \, dz-\int_{B_{r_0/\la}'} y^{1-2s}\frac{\alpha}{|x|_k^2}V^\la W  \, dz
=c_{N,s}\la^{2s}\int_{ B_{r_0/\la}^+} g(\la \cdot)\Tr(V^\la)  \Tr(W) \, dx
\end{equation}
for any $W \in H^1_{0,S_{r_0/\la}^+}(B_{r_0/\la}^+,y^{1-2s})$ (see \eqref{def_H10Br}).
Furthermore by \eqref{def_H} and a change of variables 
\begin{equation}\label{eq_blow_up_solution=1_S_1}
\int_{\mathbb{S}^+} \theta_{N+1}^{1-2s} |V^\la(\theta)|^2 dS=1 \quad \text{ for any } \la \in (0,r_0].
\end{equation}
Since the frequency function $\mc{N}$ is bounded on $[0,r_0]$ (see \eqref{ineq_N_lower_estimate} and \eqref{ineq_N_upper_estimate}) we can prove the following proposition. 
\begin{proposition}\label{prop_Vla_bounded}
The family of functions $\{V^\la\}_{\la \in (0,r_0]}$ is bounded in $H^1(B_1^+,y^{1-2s})$. 
\end{proposition}
\begin{proof}
For any $\la \in (0,r_0)$, thanks to \eqref{ineq_nabla_D+H}, \eqref{def_blow_up_solution} and a change of variables,
\begin{multline}\label{dim_Vla_bounded:1}
\mc{N}(\la)=\frac{{\la^{2s-N}}}{H(\la)}\left(\int_{B_\la^+} y^{1-2s}\left(|\nabla U|^2-\frac{\alpha}{|x|^2_k}U^2\right) \, dz-c_{N,s}\int_{B_\la'}g|\Tr(U)|^2 \, dx\right)\\
\ge \left(1-\alpha\left(\frac{2}{k-2}\right)^2+  c_{N,s}k_{N,s,g}r_0^\e\right)\frac{{\la^{2s-N}}}{H(\la)}\left(\int_{B_{\la}^+} y^{1-2s} |\nabla U|^2 \, dz\right)-\frac{N-2s}{2}\\
=\left(1-\alpha\left(\frac{2}{k-2}\right)^2+  c_{N,s}k_{N,s,g}r_0^\e\right)\left(\int_{B_{1}^+} y^{1-2s} |\nabla V^\la|^2 \, dz\right)-\frac{N-2s}{2}.
\end{multline}
Hence the claim follows from \eqref{ineq_N_upper_estimate}, \eqref{eq_blow_up_solution=1_S_1} and \eqref{ineq_hardy_with_boundary}.
\end{proof}

Now we establish the following doubling property.
\begin{proposition}\label{prop_doubling}
There exists a constant $C_3>0$ such that 
\begin{align}
&\frac{1}{C_3}H(\la)\le H(R \la) \le C_3 H(\la),\label{ineq_doubling_H}\\
&\int_{B_R^+}y^{1-2s} |V^\la|^2 \, dz \le C_3 2^{N+2-2s}\int_{B_1^+}y^{1-2s} |V^{R\la}|^2 \, dz, \label{ineq_doubling_Vla}\\
&\int_{B_R^+}y^{1-2s} |\nabla V^\la|^2 \, dz \le C_3 2^{N-2s}\int_{B_1^+}y^{1-2s} |\nabla V^{R\la}| \, dz, \label{ineq_doubling_nabla}
\end{align}
for any $\la \in (0,r_0)$ and any $R \in [1,2]$.
\end{proposition}
\begin{proof}
By \eqref{eq_H'}, \eqref{ineq_N_lower_estimate}, and \eqref{ineq_N_upper_estimate} 
\begin{equation}\label{dim_prop_doubling:1}
-\frac{N-2s}{r} \le \frac{H'(r)}{H(r)}=\frac{2\mc{N}(r)}{r} \le \frac{2C_1}{r}\quad \text{ for a.e. } r \in (0,r_0).
\end{equation}
An integration over $(\la,R \la )$ with  $R \in (1,2]$ yields
\begin{equation}\label{dim_prop_doubling:2}
R^{2s-N}\le \frac{H(R\la)}{H(\la)} \le R^{2C_1}
\end{equation}
thus \eqref{ineq_doubling_H} holds for $R \in (1,2]$ while if $R=1$ it is obvious. 

Furthermore for any $\la \in (0,r_0)$, by \eqref{ineq_doubling_H} and a change of variables,
\begin{multline}\label{dim_prop_doubling:3}
\int_{B_R^+}y^{1-2s} |V^\la|^2 \, dz=\frac{\la^{-N-2+2s}}{H(\la)}\int_{B_{R\la}^+} y^{1-2s}|U|^2 \, dz\le C_3\frac{\la^{-N-2+2s}}{H(\la R)}\int_{B_{R\la}^+} y^{1-2s}|U|^2 \, dz\\
= C_3R^{N+2-2s}\int_{B_1^+} y^{1-2s}|V^{\la R}|^2 \, dz \le C_32^{N+2-2s}\int_{B_1^+} y^{1-2s}|V^{\la R}|^2 \, dz,
\end{multline}
for any  $R \in [1,2]$. Hence we have proved \eqref{ineq_doubling_Vla} and \eqref{ineq_doubling_nabla} follows from \eqref{ineq_doubling_H} in the same way.
\end{proof}

In view of  the Coarea Formula, there exists a subset  $\mc{M}\subset (0,r_0)$ of Lebesgue measure $0$ such that $|\nabla U|\in L^2(S_r^+,y^{1-2s})$ and \eqref{eq_extension_with_boundary} holds for any $r \in  (0,r_0)\setminus \mc{M}$.

\begin{proposition}\label{prop_ineq_boundary_nabla_Vla}
There exist  $M>0$ and $\la_0>0$ such that for any $\la \in (0,\la_0)$ there exists  $R_\la \in [1,2]$ such that $R_\la\la \not \in \mc{M}$ and 
\begin{equation}\label{ineq_boundary_nabla_Vla}
\int_{S^+_{R_\la}}y^{1-2s} |\nabla V^\la|^2 \, dS \le M \int_{B^+_{R_\la}}y^{1-2s}(|\nabla V^\la|^2+|V^\la|^2) \, dz.
\end{equation} 
\end{proposition}
\begin{proof}
By Proposition \ref{prop_Vla_bounded} $\{V^\la\}_{\la \in \left(0,\frac{r_0}{2}\right)}$ is bounded in $H^1(B_2^+,y^{1-2s})$. Hence 
\begin{equation}\label{dim_ineq_boundary_nabla_Vla:1}
\limsup_{\la \to 0^+} \int_{B_2^+} y^{1-2s} (|\nabla V^\la|^2 +|V^\la|^2) \, dz < +\infty.
\end{equation}
Let, for any $\la \in \left(0,\frac{r_0}{2}\right)$,
\begin{equation}\label{dim_ineq_boundary_nabla_Vla:2}
f_\la(R):=\int_{B_R^+} y^{1-2s} (|\nabla V^\la|^2 +|V^\la|^2) \, dz.
\end{equation}
The function $f$ is absolutely continuous on $[1,2]$ and, thanks to the Coarea Formula, its distributional derivative is given by
\begin{equation}\label{dim_ineq_boundary_nabla_Vla:3}
f'_\la(R)=\int_{S_R^+} y^{1-2s} (|\nabla V^\la|^2 +|V^\la|^2) \, dS \quad \text{ for a.e. } R \in [1,2].
\end{equation}
We argue by contradiction supposing that for any $M>0$ there exists  $\la_n \to 0^+$ such that 
\begin{equation}\label{dim_ineq_boundary_nabla_Vla:4}
\int_{S_R^+} y^{1-2s}(|\nabla V^{\la_n}|^2 +|V^{\la_n}|^2)\, dS > M \int_{B_R^+}y^{1-2s}(|\nabla V^{\la_n}|^2 +|V^{\la_n}|^2) \, dz
\end{equation}
for any $n \in \mathbb{N}$ and any $R \in [1,2] \setminus {\frac{1}{\la_n}}\mc{M}$, hence for a.e. $R\in [1,2]$. 
Therefore 
\begin{equation}\label{dim_ineq_boundary_nabla_Vla:5}
f'_{\la_n}(R) > M f_{\la_n}(R) \quad \text{ for a.e. } R\in [1,2] \text{ and any } n \in \mathbb{N}.
\end{equation}
An integration over $[1,2]$ yields $f_{\la_n}(2)>e^Mf_{\la_n}(1)$ for any $n \in \mathbb{N}$. Hence 
\begin{equation}\label{dim_ineq_boundary_nabla_Vla:6}
\liminf_{n \to \infty} f_{\la_n}(1)\le \limsup_{n \to \infty} f_{\la_n}(1)\le e^{-M}\limsup_{n \to \infty} f_{\la_n}(2)
\end{equation}
and so 
\begin{equation}\label{dim_ineq_boundary_nabla_Vla:7}
\liminf_{\la \to 0^+} f_{\la}(1)\le e^{-M}\limsup_{\la \to 0^+} f_{\la}(2)
\end{equation}
for any $M>0$. It follows that $\liminf_{\la \to 0^+} f_{\la}(1)=0$ by \eqref{dim_ineq_boundary_nabla_Vla:1}. We conclude that there exists a sequence  $\la_n \to 0^+$ as $n\to \infty$ and $V \in H^1(B_1^+,y^{1-2s})$ such that 
\begin{equation}\label{dim_ineq_boundary_nabla_Vla:8}
\lim_{n \to \infty} \int_{B_1^+}y^{1-2s} (|\nabla V^{\la_n}|^2 +|V^{\la_n}|^2)\, dz=0
\end{equation}
and  $V_{\la_n} \rightharpoonup V$ weakly in $ H^1(B_1^+,y^{1-2s})$, taking into account Proposition \ref{prop_Vla_bounded}.
By Proposition \ref{prop_trace_Sr+}, \eqref{eq_blow_up_solution=1_S_1} and the lower semicontinuity of norms, we obtain  
\begin{equation}\label{dim_ineq_boundary_nabla_Vla:9}
\int_{B_1^+}y^{1-2s}(|\nabla V|^2 +|V|^2)\, dz=0 \quad \text{ and  } \quad \int_{\mathbb{S}^+}\theta_{N+1}^{1-2s}  V^2 \, dS=1
\end{equation}
which is a contradiction.
\end{proof}

\begin{proposition}\label{prop_Vla_nabla_boundary_bounded}
Let $R_\la$ be as in Proposition \ref{prop_ineq_boundary_nabla_Vla}. Then there exists a constant $\overline{M}>0$ such that 
\begin{equation}\label{ineq_Vla_nabla_boundary_bounded}
\int_{\mathbb{S}^+}\theta_{N+1}^{1-2s} |\nabla V^{R_\la \la}|^2 dS \le \overline{M} \quad \text{ for any } \la \in \left(0,\min\left\{\la_0,\frac{r_0}{2}\right\}\right).
\end{equation}
\end{proposition}
\begin{proof}
By a change of variables, the fact that $R_\la \in [1,2]$ and \eqref{def_blow_up_solution} 
\begin{multline}\label{dim_Vla_nabla_boundary_bounded:1}
\int_{\mathbb{S}^+}\theta_{N+1}^{1-2s} |\nabla V^{R_\la\la}|^2 dS=R_\la^{-N+1+2s}\frac{H(\la)}{H(R_\la\la)}\int_{S_{R_\la}^+}y^{1-2s}|\nabla V^{\la}|^2 dS\\
\le2 C_3 M \int_{B_{R_\la}^+}y^{1-2s}(|\nabla V^{\la}|^2+|V^{\la}|^2)\, dz \\
\le  2^{N+3-2s}C_3^2 M\int_{B_1^+}y^{1-2s}(|\nabla V^{R_\la\la}|^2+|V^{R_\la\la}|^2)\, dz \le \overline{M} <+\infty,
\end{multline}
for some $\overline M>0$, in view of Proposition \ref{prop_Vla_bounded}, \eqref{ineq_doubling_H},  \eqref{ineq_doubling_Vla}, \eqref{ineq_doubling_nabla}, and  \eqref{ineq_boundary_nabla_Vla}.
\end{proof}

\begin{proposition}\label{prop_blow_up}
Let $U$ be a non-trivial solution of \eqref{eq_extension_with_equation} and $\gamma$ be as in \eqref{limit_N}. Then
\begin{itemize}
\item[(i)] there exists $n \in \mathbb{N}\setminus\{0\}$  such that 
\begin{equation}\label{eq_gamma_eigenvalue}
\gamma=-\frac{N-2s}{2}+\sqrt{\left(\frac{N-2s}{2}\right)^2+\gamma_{\alpha,k,n}},
\end{equation} 
where $\gamma_{\alpha,k,n}$ is an eigenvalue of problem \eqref{prob_eigenvalue_sphere},
\item[(ii)] for any sequence $\la_p \to 0^+$ as $p \to \infty$ there exists a subsequence $\la_{p_q} \to 0^+$ as $q \to \infty$ and a eigenfunction $Z$ of problem \eqref{prob_eigenvalue_sphere}, corresponding to  the eigenvalue $\gamma_{\alpha,k,n}$,  such that  $\norm{Z}_{L^2(\mb{S}^+,\theta_{N+1}^{1-2s})}=1$ and  
\begin{equation}\label{limit_blow_up_subsequence}
\frac{U(\la_{p_q}z )}{\sqrt{H(\la_{p_q})}} \to |z|^{\gamma} Z\left(\frac{z}{|z|}\right) \quad \text{ strongly in }  H^1(B_1^+,y^{1-2s})\quad \text{ as } q \to \infty.
\end{equation}
\end{itemize}
\end{proposition}
\begin{proof}
Let $V^\la$ be as in \eqref{def_blow_up_solution} and $R_\la$ as in Proposition \ref{prop_ineq_boundary_nabla_Vla}. The family $\{V^{R_\la\la}\}_{\la \in 
\left(0,\min\left\{\la_0,\frac{r_0}{2}\right\}\right)}$ is bounded in $H^1(B_1^+,y^{1-2s})$, thanks to Proposition \ref{prop_Vla_bounded}. Let $\la_p \to 0^+$ as $p\to \infty$. Then there exists a 
subsequence  $\la_{p_q} \to 0^+$ as $q\to \infty$ and $V \in H^1(B_1^+,y^{1-2s})$ such that $V^{R_{\la_{p_q}}\la_{p_q}} \rightharpoonup V$ weakly in $H^1(B_1^+,y^{1-2s})$ as $q \to \infty$. 
By Proposition \ref{prop_trace_Sr+} the trace operator $\mathop{{\rm{Tr}}_{S^{+}_1}}$ is compact and so  
\begin{equation}\label{dim_blow_up:1}
\int_{\mb{S}^+} \theta_{N+1}^{1-2s} |V|^2 \, dS=1,
\end{equation}
in view of \eqref{eq_blow_up_solution=1_S_1}. Hence $V$ is non-trivial.
We claim that 
\begin{equation}\label{dim_blow_up:2}
V^{R_{\la_{p_q}}\la_{p_q}} \rightharpoonup V \quad \text{ strongly in }   H^1(B_1^+,y^{1-2s}) \text{ as } q \to \infty.
\end{equation}
For $q$ sufficiently large   $B_1^+\subseteq B^+_{r_0/(R_{\la_{p_q}}\la_{p_q})}$ and since   $R_{\la_{p_q}}\la_{p_q} \not \in \mc{M}$, where $\mc{M}$ is as in Proposition \
\ref{prop_ineq_boundary_nabla_Vla}, we have that 
\begin{multline}\label{dim_blow_up:3}
\int_{B_1^+} y^{1-2s} \left(\nabla V^{R_{\la_{p_q}}\la_{p_q}}\cdot\nabla W -\frac{\alpha}{|x|_k^2}V^{R_{\la_{p_q}}\la_{p_q}}W \right)\, dz\\
=\int_{\mb{S}^+} \theta_{N+1}^{1-2s} \pd{V^{R_{\la_{p_q}}\la_{p_q}}}{\nu}  \,W \, dS+c_{N,s}(R_{\la_{p_q}}\la_{p_q})^{2s}\int_{B_1'} g(R_{\la_{p_q}}\la_{p_q}\cdot) \Tr(V^{R_{\la_{p_q}}\la_{p_q}})\Tr(W) \, dx
\end{multline}
for any $W \in H^1(B_1^+,y^{1-2s})$, thanks to  \eqref{eq_extension_with_boundary} and a change of variables.
We will pass to the limit as $q \to \infty$ in $\eqref{dim_blow_up:3}$. To this end we observe that, for any $W \in  H^1(B_1^+,y^{1-2s})$,
\begin{multline}\label{dim_blow_up:4}
\left|\la^{2s}\int_{B_1'} g(\la\cdot) \Tr(V^\la)\Tr(W) \, dx\right|=\left|\frac{\la^{2s-N}}{H(\la)}\int_{B_\la'} g(x) \Tr(U)(x)\Tr(W)(\la x) \, dx\right| \\
\le k_{N,s,g}\frac{\la^{2s+\e-N}}{H(\la)} \left|\int_{B_\la^+}y^{1-2s}|\nabla U|^2\, dz-\int_{B_\la^+}y^{1-2s}\frac{\alpha}{|x|_k^2}|U|^2\, dz
+\frac{N-2s}{2\la}\int_{S^+_\la}y^{1-2s}|U|^2\, dS\right|^\frac{1}{2}\\
\times \left| \int_{B_\la^+}y^{1-2s}|\nabla W(\la \cdot)|^2\, dz-\int_{B_\la^+}y^{1-2s}\frac{\alpha}{|x|_k^2}W(\la \cdot)^2\, dz+\frac{N-2s}{2\la}
\int_{S^+_\la}y^{1-2s}|W(\la\cdot)|^2\, dS \right|^{\frac {1}{2}}\\
=k_{N,s,g} \la^\e \left|\int_{B_1^+}y^{1-2s}|\nabla V^\la|^2\, dz-\int_{B_1^+}y^{1-2s}\frac{\alpha}{|x|_k^2}|V^\la|^2\, dz+\frac{N-2s}{2}\right|^\frac{1}{2}\\
\times \left|\int_{B_1^+}y^{1-2s}|\nabla W|^2\, dz-\int_{B_1^+}y^{1-2s}\frac{\alpha}{|x|_k^2}W^2\, dz+\frac{N-2s}{2}\int_{\mb{S}^+}\theta_{N+1}^{1-2s}|W|^2\, dS \right|^{\frac {1}{2}}
\end{multline}
by a change of variables, the H\"older inequality,  \eqref{hp_g}, \eqref{ineq_h_D+H}, \eqref{def_blow_up_solution} and \eqref{eq_blow_up_solution=1_S_1}.
We conclude that 
\begin{equation}\label{dim_blow_up:5}
\lim_{\la \to 0^+}\left|\la^{2s}\int_{B_1'} g(\la\cdot) \Tr(V^\la)\Tr(W) \, dx\right|=0,
\end{equation}
by Proposition \ref{prop_Vla_bounded} and \eqref{ineq_hardy_with_boundary}.
Thanks to \eqref{ineq_Vla_nabla_boundary_bounded}, there  exists a function $f \in L^2(\mb{S}^+,\theta_{N+1}^{1-2s})$ such that 
\begin{equation}\label{dim_blow_up:6}
\pd{V^{R_{\la_{p_q}}\la_{p_q}}}{\nu} \rightharpoonup f \quad \text{ weakly in }  L^2(\mb{S}^+,\theta_{N+1}^{1-2s}) \text{ as } q \to \infty,
\end{equation}
up to a subsequence. Hence 
\begin{equation}\label{dim_blow_up:7}
\lim_{q \to \infty}\int_{\mb{S}^+}\theta_{N+1}^{1-2s} \pd{V^{R_{\la_{p_q}}\la_{p_q}}}{\nu}  \, W \, dS=\int_{\mb{S}^+} \theta_{N+1}^{1-2s} f W \, dS
\end{equation}
for any $W \in  H^1(B_1^+,y^{1-2s})$.
Furthermore 
\begin{equation}\label{dim_blow_up:8}
\lim_{q \to \infty}\int_{B_1^+} y^{1-2s} \left(\nabla V^{R_{\la_{p_q}}\la_{p_q}}\cdot\nabla W -\frac{\alpha}{|x|_k^2}V^{R_{\la_{p_q}}\la_{p_q}}W \right)\, dz=\int_{B_1^+} y^{1-2s} \left(\nabla V\cdot\nabla W -\frac{\alpha}{|x|_k^2}VW \right)\, dz
\end{equation}
by  Remark \ref{remark_equivalence_norm_H10Sr+}. It follows that
\begin{equation}\label{dim_blow_up:9}
\int_{B_1^+} y^{1-2s} \left(\nabla V\cdot\nabla W -\frac{\alpha}{|x|_k^2}VW \right)\, dz=\int_{\mb{S}^+} \theta_{N+1}^{1-2s} f W \, dS
\end{equation}
for any $W \in  H^1(B_1^+,y^{1-2s})$, that is $V$ is a weak solution of the problem 
\begin{equation}\label{dim_blow_up:9.1}
\begin{cases}
\dive(y^{1-2s}\nabla V)=\frac{\alpha}{|x|_k^2}V, &\text{ in } B_1^+,\\
-\lim_{y \to 0^+} y^{1-2s}\pd{V}{y}=0, &\text{ on } B_1'.
\end{cases}
\end{equation}
Furthermore testing \eqref{dim_blow_up:3} with $ V^{R_{\la_{p_q}}\la_{p_q}}$,
\begin{multline}\label{dim_blow_up:10}
\lim_{q \to \infty}\int_{B_1^+}y^{1-2s}\left(\left|\nabla V^{R_{\la_{p_q}}\la_{p_q}}\right|^2-\frac{\alpha}{|x|_k^2}\left|V^{R_{\la_{p_q}}\la_{p_q}}\right|^2 \right)\, dz\\
=\lim_{q \to \infty}\int_{\mb{S}^+}\theta_{N+1}^{1-2s}\pd{V^{R_{\la_{p_q}}\la_{p_q}}}{\nu}V^{R_{\la_{p_q}}\la_{p_q}}\,dS=\int_{\mb{S}^+} \theta_{N+1}^{1-2s} f W \, dS,
\end{multline}
thanks to \eqref{dim_blow_up:6} and the compactness of the trace operator $\mathop{{\rm{Tr}}_{S^+_1}}$, see Proposition \ref{prop_trace_Sr+}. 
Hence from Remark \ref{remark_equivalence_norm_H10Sr+} and \eqref{eq_blow_up_solution=1_S_1} we deduce \eqref{dim_blow_up:2}.
Let for any $ r\in (0,1]$
\begin{multline}\label{dim_blow_up:11}
D_q(r)=\frac{1}{r^{N-2s}}\Big(\int_{B_r^+} y^{1-2s} \left( |\nabla V^{R_{\la_{p_q}}\la_{p_q}}|^2 -\frac{\alpha}{|x|_k^2}|V^{R_{\la_{p_q}}\la_{p_q}}|^2 \right)\, dz\\
-c_{N,s}(R_{\la_{p_q}}\la_{p_q})^{2s}\int_{B_r'} g(R_{\la_{p_q}}\la_{p_q}\cdot) |\Tr(V^{R_{\la_{p_q}}\la_{p_q}})|^2 \, dx\Big)
\end{multline}
and 
\begin{equation}\label{dim_blow_up:12}
H_q(r)=\frac{1}{r^{N+1-2s}}\int_{S_r^+} y^{1-2s} |V^{R_{\la_{p_q}}\la_{p_q}}|^2\, dS.
\end{equation}
For any $r \in (0,1]$ we also define 
\begin{equation}\label{dim_blow_up:13}
D_V(r)=\frac{1}{r^{N-2s}}\int_{B_r^+} y^{1-2s} \left( |\nabla V|^2 -\frac{\alpha}{|x|_k^2}|V|^2 \right)\, dz\\
\end{equation}
and 
\begin{equation}\label{dim_blow_up:14}
H_V(r)=\frac{1}{r^{N+1-2s}}\int_{S_r^+} y^{1-2s} |V|^2\, dS.
\end{equation}
Thanks to a scaling argument it is easy to see that 
\begin{equation}\label{dim_blow_up:15}
\mc{N}_q(r):=\frac{D_q(r)}{H_q(r)}=\frac{D(R_{\la_{p_q}}\la_{p_q}r)}{H(R_{\la_{p_q}}\la_{p_q}r)}=\mc{N}(R_{\la_{p_q}}\la_{p_q}r) \quad \text{ for any } r \in (0,1].
\end{equation} 
By \eqref{dim_blow_up:2},  \eqref{dim_blow_up:5} and Remark \ref{remark_equivalence_norm_H10Sr+}, it  follows that 
\begin{equation}\label{dim_blow_up:16}
H_q(r) \to H_V(r) \quad \text{ and } \quad D_q(r) \to D_V(r), \quad \text{ as } q \to \infty, \text{ for any } r \in (0,1].
\end{equation}
Furthermore $H_V(r)>0$ for any $r \in (0,1]$ by Proposition \ref{prop_H_not_0} in the case $g\equiv 0$ and $\Omega = B_2'$.
In particular the function 
\begin{equation}\label{dim_blow_up:17}
\mc{N}:(0,1] \to \R, \quad \mc{N}_V(r):=\frac{D_V(r)}{H_V(r)}
\end{equation}
is well defined and $\mc{N}_V \in W^{1,1}_{loc}((0,1])$ by Proposition \ref{prop_N'} in the case $g\equiv 0$ and $\Omega = B_2'$.
In view of \eqref{dim_blow_up:14}, \eqref{limit_N}
\begin{equation}\label{dim_blow_up:18}
\mc{N}_V(r)=\lim_{q \to \infty} \mc{N}(R_{\la_{p_q}}\la_{p_q}r)=\gamma  \quad \text{ for any } r \in (0,1].
\end{equation}
Hence $\mc{N}_V(r)$ is constant in $[0,1]$ and so 
\begin{equation}\label{dim_blow_up:19}
\mc{N}'_V(r)\equiv 0 \text{ for any } r \in (0,1].
\end{equation}
By Proposition \ref{prop_N'} it follows that 
\begin{equation}\label{dim_blow_up:20}
\left(\int_{S_r^+}y^{1-2s}V^2\, dS\right)\left(\int_{S_r^+}y^{1-2s}\left|\pd{V}{\nu}\right|^2 \, dS\right)-\left(\int_{S_r^+}y^{1-2s}V \pd{V}{\nu} \, dS\right)^2=0
\end{equation}
for a.e. $r \in (0,1)$, that is, equality holds in the Cauchy-Schwartz inequality for the vectors $V$ and $\pd{V}{\nu}$ in $L^2(S_r^+,y^{1-2s})$ for  a.e. $r \in (0,1)$.
Therefore there exists a function $\eta(r)$ defined a.e. in $(0,1)$ such that 
\begin{equation}\label{dim_blow_up:21}
\pd{V}{\nu}(r\theta)=\eta(r) V(r\theta) \quad  \text{ for a.e. r } \in (0,1) \text{ and a.e.  } \theta \in \mb{S}^+.
\end{equation}
 Multiplying by $V(r\theta)$ and integrating over $\mb{S}^+$,
\begin{equation}\label{dim_blow_up:22}
\int_{\mb{S}^+}\theta_{N+1}^{1-2s}\pd{V}{\nu}(r\theta)V(r\theta) \, dS=\eta(r)\int_{\mb{S}^+}\theta_{N+1}^{1-2s}|V(r\theta)|^2 \, dS \quad  \text{ for a.e. r } \in (0,1)
\end{equation}
 and so 
$\eta(r)=\frac{H'_V(r)}{2H_V(r)}=\frac{\gamma}{r}$ for a.e. $r  \in (0,1)$ by \eqref{eq_H'}, \eqref{eq_H'} and \eqref{dim_blow_up:18}. Since $V$ is smooth away from $\Sigma_k$  by classical elliptic regularity theory (see \eqref{def_Sigma}), an integration over $(r,1)$ yields 
\begin{equation}\label{dim_blow_up:23}
V(r\theta)=r^\gamma V(1\theta)=r^\gamma Z(\theta) \quad  \text{ for any r } \in (0,1] \text{ and a.e.  } \theta \in \mb{S}^+, 
\end{equation}
where $Z=V_{|\mb{S}^+}$ and  $\norm{Z}_{L^2(\mb{S}^+,\theta^{1-2s}_{N+1})}=1$ by \eqref{dim_blow_up:1} . In view of \cite[Lemma 1.1]{FF}, \eqref{dim_blow_up:23} and \eqref{dim_blow_up:9.1}
the function $Z$ is an eigenfunction of problem \eqref{prob_eigenvalue_sphere} and the correspondent eigenvalue $\gamma_{\alpha,k,n}$ satisfies the relationship $\gamma(N-2s+\gamma)=  
\gamma_{\alpha,k,n}$,  that is   
\begin{equation}\label{dim_blow_up:24}
\gamma =-\frac{N-2s}{2}+\sqrt{\left(\frac{N-2s}{2}\right)^2 +\gamma_{\alpha,k,n}} \;\;\text{ or } \;\;\gamma =-\frac{N-2s}{2} -\sqrt{\left(\frac{N-2s}{2}\right)^2 +\gamma_{\alpha,k,n}}
\end{equation}
Since $r^\gamma Z(\theta) \in H^1(B_1^+,y^{1-2s})$ by \eqref{dim_blow_up:23} then $r^{2\gamma-2} Z^2(\theta) \in L^1(B_1^+,y^{1-2s})$ by \eqref{ineq_hardy_with_boundary} and so   we conclude that 
\eqref{eq_gamma_eigenvalue} must hold.

Consider now the sequence $\{V^{\la_{p_q}}\}_{q \in \mathbb{N}}$. Up to a further subsequence, $V^{\la_{p_q}} \rightharpoonup \tilde{V}$  weakly in $H^1(B_1^+,y^{1-2s})$ as $q \to \infty$, for some  
$\tilde{V} \in H^1(B_1^+,y^{1-2s})$ and $R_{\la_{p_q}} \to \tilde{R}$, for some $\tilde{R} \in [1,2]$ as $q \to \infty$. The strong convergence of  $\{V^{R_{\la_{p_q}}\la_{p_q}}\}_{q \in \mathbb{N}}$ 
to $V$ in $H^1(B_1^+,y^{1-2s})$ implies that, up to a further subsequence, both $V^{R_{\la_{p_q}}\la_{p_q}}$ and $\left|\nabla V^{R_{\la_{p_q}}\la_{p_q}}\right|$ are dominated a.e. by a 
$L^2(B_1^+,y^{1-2s})$ function, uniformly with respect to $q \in \mathbb{N}$. Up to a further subsequence, we may also assume that the limit 
\begin{equation}\label{dim_blow_up:25}
\ell =\lim_{q\to \infty}\frac{H(R_{\la_{p_q}}\la_{p_q})}{H(\la_{p_q})}
\end{equation}
exists, it is finite and strictly positive, taking into account \eqref{ineq_doubling_H}. Then from the Dominated Convergence Theorem  and a change of variables we deduce that 
\begin{multline}\label{dim_blow_up:26}
\lim_{q\to \infty}\int_{B_1^+} y^{1-2s} V^{\la_{p_q}}(z) \phi(z) \, dz= 
\lim_{q\to \infty}R_{\la_{p_q}}^{N+2-2s}\int_{B_{1/R_{\la_{p_q}}}^+} y^{1-2s}V^{\la_{p_q}}(R_{\la_{p_q}}z)\phi(R_{\la_{p_q}} z) \, dz\\
=\lim_{q\to \infty}R_{\la_{p_q}}^{N+2-2s}\sqrt{\frac{H(R_{\la_{p_q}}\la_{p_q})}{H(\la_{p_q})}}
\int_{B_{1}^+} y^{1-2s}\chi_{B_{1/R_{\la_{p_q}}}^+}(z)V^{R_{\la_{p_q}}\la_{p_q}}(z) \phi(R_{\la_{p_q}} z) \, dz\\
=\tilde{R}^{N+2-2s} \sqrt{\ell} \int_{B_{1/\tilde{R}}^+} y^{1-2s}V(z) \phi(\tilde{R} z) \, dz= \sqrt{\ell} \int_{B_1^+} y^{1-2s}V(z/\tilde{R}) \phi( z) \, dz
\end{multline}
for any $\phi \in C^\infty(\overline{B_1^+})$. By density we conclude that $V^{\la_{p_q}} \rightharpoonup \sqrt{\ell}V(\cdot/\tilde{R})$ weakly in  $L^2(B_1^+,y^{1-2s})$ as $q \to \infty$.
Since $V^{\la_{p_q}} \rightharpoonup \tilde{V}$ weakly in  $H^1(B_1^+,y^{1-2s})$ as $q \to \infty$ we conclude that $\tilde{V}=\sqrt{\ell}V(\cdot/\tilde{R})$ and so $V^{\la_{p_q}} \rightharpoonup \sqrt{\ell}V(\cdot/\tilde{R})$ weakly in  $H^1(B_1^+,y^{1-2s})$ as $q \to \infty$. Furthermore
\begin{multline}\label{dim_blow_up:27}
\lim_{q\to \infty} \int_{B_1^+}y^{1-2s} |\nabla V^{\la_{p_q}}(z)|^2 \, dz= 
\lim_{q\to \infty}R_{\la_{p_q}}^{N+2-2s}\int_{B_{1/R_{\la_{p_q}}}^+} y^{1-2s} |\nabla V^{\la_{p_q}}(R_{\la_{p_q}} z)|^2\, dz\\
=\lim_{q\to \infty}R_{\la_{p_q}}^{N-2s}\frac{H(R_{\la_{p_q}}\la_{p_q})}{H(\la_{p_q})}
\int_{B_{1}^+} y^{1-2s}\chi_{B_{1/R_{\la_{p_q}}}^+}(z)|\nabla V^{R_{\la_{p_q}}\la_{p_q}}(z)|^2 \, dz\\
=\tilde{R}^{N-2s} \ell \int_{B_{1/\tilde{R}}^+} y^{1-2s}|\nabla V|^2 dz=  \int_{B_1^+} y^{1-2s}|\sqrt{\ell}\nabla V(\cdot/\tilde{R})|^2\, dz,
\end{multline}
by the Dominated Convergence Theorem and a change of variables. Hence $V^{\la_{p_q}} \to \sqrt{\ell}V(\cdot/\tilde{R})$ strongly in  $H^1(B_1^+,y^{1-2s})$ as $q \to \infty$.

Thanks to \eqref{dim_blow_up:23}, $V$ is a homogeneous function of degree $\gamma$ and so $\tilde{V}=\sqrt{\ell}\tilde{R}^{-\gamma} V$. Moreover, since $V^{\la_{p_q}} \to \tilde{V}$ strongly in  $L^2(\mathbb{S}^+,\theta_{N+1}^{1-2s})$ as $q \to \infty$ by Proposition \ref{prop_trace_Sr+}, 
\begin{equation}\label{dim_blow_up:28}
1=\int_{\mathbb{S}^+} \theta_{N+1}^{1-2s} |\tilde{V}(\theta)|^2 dS=\sqrt{\ell}\tilde{R}^{-\gamma}\int_{\mathbb{S}^+} \theta_{N+1}^{1-2s} |V(\theta)|^2 dS=\sqrt{\ell}\tilde{R}^{-\gamma}
\end{equation}
in view of  \eqref{eq_blow_up_solution=1_S_1} and  \eqref{dim_blow_up:1}.
We conclude that $\tilde{V} =V$ thus completing the proof. 
\end{proof}

Now we  show that the limit \eqref{limit_H_exists_finite} is strictly positive, by means  of a Fourier analysis with  respect to the $L^2(\mb{S}^+,\theta_{N+1}^{1-2s})$-orthonormal basis 
$\{Z_{\alpha,k,n}\}_{n\in \mb{N}\setminus\{0\}}$ of  eigenfunctions of problem \eqref{prob_eigenvalue_sphere}, see Subsection \ref{subsec_an_eigenvalue_problem_on_S+}. To this end let us define for any 
$k \in \{3, \dots,N\}$, $\alpha$ as in \eqref{hp_alpha}, and $n \in \mathbb{N} \setminus\{0\}$
\begin{equation}\label{def_fourier_coefficents}
\varphi_{n,i}(\la):=\int_{\mathbb{S}^+} \theta_{N+1}^{1-2s} U(\la \theta)Z_{\alpha,k,n,i}(\theta) \, dS, \quad \text{ for any } \la \in (0,r_0], \;i\in{1,\dots,M_{\alpha,k,n}},
\end{equation}
see \eqref{def_dimension-eigenspace} for the definition of $M_{\alpha,k,n}$, and 
\begin{equation}\label{def_fourier_coefficents_g}
\Upsilon_{n,i}(\la):=c_{N,s} \int_{B'_\la} g \Tr(U) \Tr\left(Z_{\alpha,k,n,i}\left(\frac{\cdot}{|\cdot|}\right)\right)\, dx,    .
\end{equation}
for any $\la \in (0,r_0], \; i\in{1,\dots,M_{\alpha,k,n}}$. Thanks to  Proposition \ref{prop_limit_N} and Proposition \ref{prop_blow_up} there exists $n_0 \in \mathbb{N} \setminus \{0\}$ such that 
\begin{equation}\label{eq_gamma_limit_eigenvalue}
\gamma=\lim_{r \to 0^+} \mathcal{N}(r)=-\frac{N-2s}{2}+\sqrt{\left(\frac{N-2s}{2}\right)^2+\gamma_{\alpha,k,n_0}}.
\end{equation}
For any $i \in\{1,\dots, M_{\alpha,k,n_0}\}$ we need to compute the asymptotics of $\varphi_{n_0,i}(\la)$ as $\la \to 0^+$.

\begin{proposition} \label{prop_fourier_coefficents}
Let ${n_0}$ be as in  \eqref{eq_gamma_limit_eigenvalue}. Then for any $i \in\{1,\dots, M_{\alpha,k,n_0}\}$ and any $r \in (0,r_0]$
\begin{multline}\label{eq_asymptotic_fourier}
\varphi_{n_0,i}(\la)=\la^{\gamma}\Bigg(\frac{\varphi_{n_0,i}(r)}{r^{\gamma}}
+\frac{\gamma r^{-N+2s-2\gamma}}{N-2s+2\gamma}\int_{0}^{r}\rho^{-1+\rho}\Upsilon_{n_0,i}(\rho) d\rho\\
+\frac{N-2s+\gamma}{N-2s+2 \gamma}\int_{\la}^{r}\rho^{-N-1+2s-\gamma} \Upsilon_{n_0,i}(\rho )\, d\rho\Bigg)+ O(\la^{\gamma+\e})\quad \text{ as } \la \to 0^+.
\end{multline}
\end{proposition}
\begin{proof}
Let $n \in \mathbb{N}$ and $i \in \{1,\dots,M_{\alpha,k,n}\}$. Let  $f \in C_c^\infty(0,r_0)$. Then testing \eqref{eq_extension_with_equation} with the function $|z|^{N+1-2s}f(|z|) Z_{\alpha,k,n,i}(z/|z|)$ and 
passing in polar coordinates, by \eqref{eq_egienvalue_sphere}, we obtain 
\begin{equation}\label{dim_fourier_coefficents:1}
-\varphi_{n,i}''(\la)-\frac{N+1-2s}{\la }\varphi_{n,i}'(\la)+\frac{\gamma_{\alpha,k,n}}{\la^2}\varphi_{n,i}(\la)=\zeta_{n,i}(\la)   \quad  \text{ in } (0,r_0)
\end{equation}
in a distributional sense, where the distribution $\zeta_{n,i} \in \mathcal{D}'(0,r_0)$ is define as
\begin{equation}\label{dim_fourier_coefficents:2}
\df{\mathcal{D}'(0,r_0)}{\zeta_{n,i}}{f}{\mathcal{D}(0,r_0)}= 
\int_{0}^{r_0} \frac{f(\la)}{\la^{2-2s}} \left( \int_{\mb{S'}} g(\la \cdot ) \Tr(U)(\la \cdot ) \Tr\left(Z_{\alpha,k,n,i}\left(\frac{\cdot}{|\cdot|}\right)\right)\, dS'\right) \, d \la,
\end{equation}
for any $f \in C_c^\infty(0,r_0)$. In particular $\zeta_{n,i}$ belongs to $L^1_{loc}((0,r_0])$ by  the Coarea Formula and a change of variables. If $\Upsilon_{n,i} $ is as in \eqref{def_fourier_coefficents_g}, a direct computation shows that 
\begin{equation}\label{dim_fourier_coefficents:3}
\Upsilon'_{n,i}(\la)=\la^{N+1-2s}\zeta_{n,i}(\la)   \quad  \text{ in } \mathcal{D}'(0,r_0)
\end{equation}
hence 
\begin{equation}\label{dim_fourier_coefficents:4}
-\left(\la^{N+1-2s+2\sigma_n}\left(\la^{-\sigma_n} \varphi_{n,i}(\la)\right)'\right)'=\la^{\sigma_n} \Upsilon'_{n,i}(\la) \quad  \text{ in } \mathcal{D}'(0,r_0),
\end{equation}
where
\begin{equation}\label{dim_fourier_coefficents:5}
\sigma_n:=-\frac{N-2s}{2}+\sqrt{\left(\frac{N-2s}{2}\right)^2+\gamma_{\alpha,k,n.}}
\end{equation}
From \eqref{dim_fourier_coefficents:4} and \eqref{dim_fourier_coefficents:2} we deduce that $\la \to \la^{N+1-2s+2\sigma_n}\left(\la^{-\sigma_n} \varphi_{n,i}'(\la)\right)$ belongs to 
$W^{1,1}_{loc}((0,r_0])$ hence an integration over $(\la, r)$ yields  
\begin{multline}\label{dim_fourier_coefficents:6}
\left(\la^{-\sigma_n} \varphi_{n,i}(\la)\right)'=-\la^{-N-1+2s-\sigma_n}\Upsilon_{n,i}(\la)\\
-\la^{-N-1+2s-2\sigma_n}\sigma_n\left(C(r)+\int_{\la}^r\rho^{\sigma_n-1}\Upsilon_{n,i}(\rho) \, d\rho\right)
\end{multline}
for any $ r\in (0,r_0]$, for some real number $C(r)$ depending on $r, \alpha,k,n$ and $i$.
Since in view of  \eqref{dim_fourier_coefficents:6} $\la \to \la^{-\sigma_n} \varphi_{n,i}(\la)$ belongs to  $W^{1,1}_{loc}((0,r_0])$, a  further integration  yields
\begin{multline}\label{dim_fourier_coefficents:7}
\varphi_{n,i}(\la)=\la^{\sigma_n}\Bigg(r^{-\sigma_n}\varphi_{n,i}(r)+ \int_{\la}^{r}\rho^{-N-1+2s-\sigma_n}\Upsilon_{n,i}(\rho )\, d \rho\\
+\sigma_n\int_{\la}^r\rho^{-N-1+2s-2\sigma_n}\left(C(r)+\int_{\rho}^r t^{\sigma_n-1}\Upsilon_{n,i}(t) \, dt\right)\, d \rho\Bigg)\\
=\la^{\sigma_n}\Bigg(r^{-\sigma_n}\varphi_{n,i}(r)+ \int_{\la}^{r}\rho^{-N-1+2s-\sigma_n}\Upsilon_{n,i}(\rho )\, d\rho
+\frac{\sigma_nC(r)r^{-N+2s-2\sigma_n}}{-N+2s-2\sigma_n}\\
-\frac{\sigma_nC(r)\la^{-N+2s-2\sigma_n}}{-N+2s-2\sigma_n}-\frac{\sigma_n \la^{-N+2s-2\sigma_n}}{-N+2s-2\sigma_n}\int_{\la}^r t^{\sigma_n-1}\Upsilon_{n,i}(t) \, dt\\ +\frac{\sigma_n}{-N+2s-2\sigma_n}\int_{\la}^r\rho^{-N-1+2s-\sigma_n}\Upsilon_{n,i}(\rho )\, d\rho \Bigg)\\
=\la^{\sigma_n}\Bigg(\frac{\varphi_{n,i}(r)}{r^{\sigma_n}}-\frac{\sigma_nC(r)r^{-N+2s-2\sigma_n}}{N-2s+2\sigma_n}
+\frac{N-2s+\sigma_n}{N-2s+2 \sigma_n}\int_{\la}^{r}\rho^{-N-1+2s-\sigma_n}\Upsilon_{n,i}(\rho )\, d\rho\Bigg)\\
+\frac{\sigma_n\la^{-N+2s-\sigma_n}}{N-2s+2\sigma_n}\left(C(r)+\int_{\la}^r t^{\sigma_n-1}\Upsilon_{n,i}(t) \, dt\right)
\end{multline}
for any $\la \in (0,r_0]$.

Let $n_0$ be as in \eqref{eq_gamma_limit_eigenvalue} and $i  \in \{1,\dots,M_{\alpha,k,n_0}\}$. By \eqref{eq_gamma_limit_eigenvalue} and \eqref{dim_fourier_coefficents:5}, $\gamma=\sigma_{n_0}$ and 
\begin{multline}\label{dim_fourier_coefficents:8}
\la^{-N-1+2s-\gamma}\left|\Upsilon_{n_0,i}(\la)\right|
<c_{N,s}\la^{-N-1+2s-\gamma} \int_{B'_\la} |g| |\Tr(U)\ \left|\Tr\left(Z_{\alpha,k,n,i}\left(\frac{\cdot}{|\cdot|}\right)\right)\right|\, dx \\
\le \la^{-N-1+2s-\gamma}\left(\int_{B'_\la} |g| |\Tr(U)|^2 dx \right)^\frac12
\left(\int_{B'_\la} |g| \left|\Tr\left(Z_{\alpha,k,n,i}\left(\frac{\cdot}{|\cdot|}\right)\right)\right|^2 dx \right)^\frac12 \\
\le  k_{N,s,g}\la^{-N-1+2s-\gamma+\e}\left(\int_{B_\la^+}y^{1-2s}|\nabla U|^2\, dz-\int_{B_\la^+}y^{1-2s}\frac{\alpha}{|x|_k^2}U^2\, dz
+\frac{N-2s}{2\la}\int_{S_\la^+}y^{1-2s}U^2\, dz\right)^\frac12\\
\times \Bigg(\int_{B_\la^+}y^{1-2s}|\nabla Z_{\alpha,k,n,i}(z/|z|)|^2\, dz-\int_{B_\la^+}y^{1-2s}\frac{\alpha}{|x|_k^2} |Z_{\alpha,k,n,i}(z/|z|)|^2\, dz \\
+\frac{N-2s}{2\la}\int_{S_\la^+}y^{1-2s} |Z_{\alpha,k,n,i}(z/|z|)|^2\, dz\Bigg)^\frac12 \\
=k_{N,s,g}\la^{-1-\gamma+\e}\sqrt{H(\la)}\left(\int_{B_1^+}y^{1-2s}|\nabla V^\la|^2\, dz-\int_{B_1^+}y^{1-2s}\frac{\alpha}{|x|_k^2}|V^\la|^2\, dz+\frac{N-2s}{2}\right)^\frac12\\
\times\left(\int_{B_1^+}y^{1-2s}|\nabla  Z_{\alpha,k,n,i}(z/|z|)|^2\, dz-\int_{B_1^+}y^{1-2s}\frac{\alpha}{|x|_k^2}| Z_{\alpha,k,n,i}(z/|z|)|^2\, dz+\frac{N-2s}{2}\right)^\frac12\\
\le {\rm{const}}\;\la^{-1+\e}
\end{multline}
for any $\la \in (0,r_0]$, by Holder inequality, a change of variables, \eqref{hp_g}, \eqref{ineq_h_D+H},  \eqref{ineq_H_upper_estimate},  \eqref{def_blow_up_solution}, 
\eqref{eq_blow_up_solution=1_S_1}, \eqref{def_fourier_coefficents_g}.
Hence 
\begin{equation}\label{dim_fourier_coefficents:9}
\left|\Upsilon_{n_0,i}(\la)\right|\le {\rm{const}}\ \la^{N-2s+\gamma+\e} \quad  \text{ for any } \la \in (0,r_0].
\end{equation}
Now we show that for any $r \in (0,r_0]$
\begin{equation}\label{dim_fourier_coefficents:12}
C(r)+\int_{0}^{r}\la^{-1+\gamma}\Upsilon_{n_0,i}(\la) d\la=0.
\end{equation}
From \eqref{dim_fourier_coefficents:9} it is clear that $\int_{0}^{r_0}\la^{-1+\gamma}\Upsilon_{n_0,i}(\la) d\la<+\infty$. We argue by contradiction. Since 
$\sigma_{n_0}=\gamma > -\frac{N-2s}{2}$ by \eqref{eq_gamma_limit_eigenvalue} and \eqref{dim_fourier_coefficents:5}, then from  \eqref{dim_fourier_coefficents:7} we deduce that 
\begin{equation}\label{dim_fourier_coefficents:13}
\varphi_{n_0,i}(\la)\sim\frac{\gamma\la^{-N+2s-\gamma}}{N-2s+2\gamma}\left(C(r)+\int_{\la}^r t^{-1+\gamma}\Upsilon_{n_0,i}(t) \, dt\right) \quad  \text{ as } \la \to 0^+
\end{equation}
and so by \eqref{eq_gamma_limit_eigenvalue}
\begin{equation}\label{dim_fourier_coefficents:14}
\int_{0}^{r_0}\la^{N-1-2s}|\varphi_{n_0,i}(\la)|^2 \, d \la =+\infty.
\end{equation}
On the other hand by H\"older inequality, a change of variables, \eqref{def_fourier_coefficents} and \cite[Lemma 2.4] {FF}
\begin{multline}\label{dim_fourier_coefficents:15}
\int_{0}^{r_0}\la^{N-1-2s}|\varphi_{n_0,i}(\la)|^2 \, d \la \le\int_{0}^{r_0}\la^{N-1-2s}\left(\int_{\mb{S}^+}\theta_{N+1}^{1-2s} |U(\la\theta)|^2 \,dS \right) d \la \\
= \int_{B^+_{r_0}} y^{1-2s} \frac{U^2}{|z|^2} \, dz <+\infty,
\end{multline}
which contradicts \eqref{dim_fourier_coefficents:14}. It follows that 
\begin{multline}\label{dim_fourier_coefficents:16}
\la^{-N+2s-\gamma}\left|C(r)+\int_{\la}^{r}\la^{-1+\gamma}\Upsilon_{n_0,i}(\la) d\la\right|=\la^{-N+2s-\gamma}\left|\int_{0}^{\la}\la^{-1+\gamma}\Upsilon_{n_0,i}(\la) d\la\right|\\
= O(\la^{\gamma+\e}),
\end{multline}
in view of \eqref{dim_fourier_coefficents:9}. In conclusion \eqref{eq_asymptotic_fourier} follows from 
\eqref{dim_fourier_coefficents:7},  \eqref{dim_fourier_coefficents:12},  and \eqref{dim_fourier_coefficents:16}. 
\end{proof}

\begin{proposition}\label{prop_limit_H_positive}
Let $U$ be a non-trivial solution of \eqref{eq_extension_with_equation} and $\gamma$ be as in \eqref{limit_N}. Then 
\begin{equation}\label{limit_H_positive}
\lim_{r \to 0^+}r^{-2\gamma}H(r)>0.
\end{equation}
\end{proposition}
\begin{proof}
From \eqref{def_fourier_coefficents}, since $\{Z_{\alpha,k,n}\}_{n \in \mathbb{N}\setminus\{0\}}$  is a orthonormal basis of $L^2(\mb{S}^+,\theta_{N+1}^{1-2s})$, see Subsection \ref{subsec_an_eigenvalue_problem_on_S+}, we have that
\begin{equation}\label{dim_limit_H_positive:1}
H(\la)=\int_{\mb{S}^+}\theta_{N+1}^{1-2s}|U(\la \theta)|^2 \, dS=\sum_{n=1}^\infty \sum_{i=1}^{M_{\alpha,k,n}} |\varphi_{n,i}(\la)|^2
\end{equation}
by \eqref{def_H} and a change of variables.  We argue by contradiction supposing that 
\begin{equation}\label{dim_limit_H_positive:2}
\lim_{\la \to 0^+}\la^{-2\gamma}H(\la)=0.
\end{equation} 
Let $n_0$ be as in \eqref{eq_gamma_limit_eigenvalue}. By \eqref{dim_limit_H_positive:1} for any $i\in\{1,\dots,M_{\alpha,k,n_0}\}$, 
\begin{equation}\label{dim_limit_H_positive:3}
	\lim_{\la \to 0^+}\la^{-2\gamma} |\varphi_{n_0,i}(\la)|^2=0.
\end{equation} 
By \eqref{eq_asymptotic_fourier}, for any $i\in\{1,\dots,M_{\alpha,k,n_0}\}$ and any $r \in (0,r_0]$
\begin{multline}\label{dim_limit_H_positive:4}
\frac{\varphi_{n,i}(r)}{r^{\gamma}}+\frac{\gamma r^{-N+2s-2\gamma}}{N-2s+2\gamma}\int_{0}^{r}\rho^{-1+\rho}\Upsilon_{n_0,i}(\rho) d\rho\\
+\frac{N-2s+\gamma}{N-2s+2 \gamma}\int_{0}^{r}\rho^{-N-1+2s-\gamma}\Upsilon_{n,i}(\rho )\, d\rho=0.
\end{multline}
Hence by \eqref{eq_asymptotic_fourier}, \eqref{dim_fourier_coefficents:9} and \eqref{dim_limit_H_positive:4}
\begin{equation}\label{dim_limit_H_positive:5}
\varphi_{n,i}(\la)=-\la^\gamma\frac{N-2s+\gamma}{N-2s+2 \gamma}\int_{0}^{\la}\rho^{-N-1+2s-\gamma}\Upsilon_{n,i}(\rho )\, d\rho+O(\la^{\gamma +\e})=O(\la^{\gamma +\e})
\end{equation}
as $\la \to 0^+$ for any $i\in\{1,\dots,M_{\alpha,k,n_0}\}$. In view of  \eqref{def_H} and \eqref{def_blow_up_solution}, it follows that 
\begin{equation}\label{dim_limit_H_positive:6}
\sqrt{H(\la)}\int_{\mathbb{S}^+} \theta_{N+1}^{1-2s} V^\la Z \, dS=O(\la^{\gamma +\e}) \quad \text{ as } \la \to 0^+,
\end{equation}
for  any $Z \in V_{n_0}$, see \eqref{def_eigenspace}. Then, in view of  \eqref{ineq_H_lower_estimate}  with $\sigma=\frac{\e}{2}$,
\begin{equation}\label{dim_limit_H_positive:7}
\int_{\mathbb{S}^+} \theta_{N+1}^{1-2s} V^\la Z \, dS=O(\la^{\frac{\e}{2}}) \quad \text{ as } \la \to 0^+
\end{equation}
for  any $Z \in V_{n_0}$.
On the other hand by Proposition \ref{prop_blow_up} and Proposition \ref{prop_trace_Sr+}, there exist $Z_0\in  V_{n_0}$ with $\norm{Z_0}_{L^2(\mb{S}^+,\theta_{N+1}^{1-2s})}=1$ and a sequence $\la_q \to 0^+$ as $q \to \infty$ such that 
\begin{equation}\label{dim_limit_H_positive:8}
\quad V^{\la_q} \to Z_0 \quad \text{ strongly in } L^2(\mb{S}^+,\theta_{N+1}^{1-2s}) \text{ as } q \to \infty.
\end{equation}
Since $Z_0 \in V_{n_0}$,  from the Parseval identity, \eqref{dim_limit_H_positive:7}, and \eqref{dim_limit_H_positive:8} we deduce that $Z_0\equiv 0$ which contradicts the fact that $\norm{Z_0}_{L^2(\mb{S}^+,\theta_{N+1}^{1-2s})}=1$.
\end{proof}

We are now in position to state and prove our main results which are a more precise version of Theorem \ref{theorem_blow_up_extension_not_precise} and Theorem \ref{theorem_blow_up_not_precise} respectively.

\begin{theorem}\label{theorem_blow_up_extension_precise}
Let $U$ be a solution of \eqref{eq_extension_with_equation} and suppose that $g$ satisfies \eqref{hp_g}. Then there exists $n \in \mathbb{N} \setminus\{0\}$ such that 
\begin{equation}\label{eq_gamma_limit_eigenvalue_theorem}
\gamma=\lim_{r \to 0^+} \mathcal{N}(r)=-\frac{N-2s}{2}+\sqrt{\left(\frac{N-2s}{2}\right)^2+\gamma_{\alpha,k,n}}.
\end{equation}
Furthermore let $M_{\alpha,k,n}$ and $\{Z_{\alpha,k,n,i}\}_{i \in \{1,\dots.M_{\alpha,k,n}\}}$ be as in \eqref{def_dimension-eigenspace} and \eqref{def_basis_eigenfunction} respectively. Then for any $i \in \{0,\dots.M_{\alpha,k,n}\}$ there exists $\beta_i\in \R$ such that 
$(\beta_1,\dots,\beta_{M_{\alpha,k,n}}) \neq (0,\dots,0)$ and 
\begin{equation}\label{eq_limit_extension_blow_up_precise}
\frac{U(\la z)}{\la^\gamma} \to |z|^\gamma\sum_{i=1}^{M_{\alpha,k,n}} \beta_i Z_{\alpha,k,n,i}(z/|z|) \quad \text{ strongly in } H^1(B_1^+,y^{1-2s}) \text { as } \la \to 0^+,
\end{equation}
where 
\begin{multline}\label{def_betai}
\beta_i:=\frac{\varphi_{n,i}(r)}{r^{\gamma}}+\frac{\gamma r^{-N+2s-2\gamma}}{N-2s+2\gamma}\int_{0}^{r}\rho^{-1+\rho}\Upsilon_{n,i}(\rho) d\rho\\
+\frac{N-2s+\gamma}{N-2s+2 \gamma}\int_{0}^{r}\rho^{-N-1+2s-\gamma}\Upsilon_{n,i}(\rho )\, d\rho \quad \text{ for any } r \in (0,r_0],
\end{multline}
with  $\varphi_{n,i}$  and $\Upsilon_{n,i}$ given by  \eqref{def_fourier_coefficents} and \eqref{def_fourier_coefficents_g} respectively.
\end{theorem}
\begin{proof}
In view of \eqref{limit_N} and Proposition \ref{prop_blow_up} we know that \eqref{eq_gamma_limit_eigenvalue_theorem} holds for some $n \in \mathbb{N}\setminus\{0\}$. Furthermore for any sequence of  strictly positive numbers $\la_p \to 0^+$  as $p \to \infty$ there exist a subsequence $\la_{p_q} \to 0^+$   as $q \to \infty$  and real numbers $\beta_1,\dots,\beta_{M_{\alpha,k,n}}$ such that 
\begin{equation}\label{dim_blow_up_extension_precise:1}
\frac{U(\la z)}{\la^\gamma} \to |z|^\gamma\sum_{i=1}^{M_{\alpha,k,n}} \beta_i Z_{\alpha,k,n,i}(z/|z|) \quad \text{ strongly in } H^1(B_1^+,y^{1-2s}) \text { as } q \to \infty^+,
\end{equation}
taking into account Proposition \ref{prop_blow_up} and \eqref{def_basis_eigenfunction}. We claim that for any $i \in \{1,\dots,M_{\alpha,k,n}\}$ the number $\beta_i$ does not depend neither 
on the sequence $\la_p \to 0^+$ nor on its subsequence $\la_{p_q}\to 0^+$.
In view of \eqref{def_basis_eigenfunction}, \eqref{def_fourier_coefficents}, \eqref{dim_blow_up_extension_precise:1} and Proposition \ref{prop_trace_Sr+}
\begin{multline}\label{dim_blow_up_extension_precise:2}
\lim_{q \to \infty}\la_{p_q}^{-\gamma}\varphi_{n,j}(\la_{p_q})=\lim_{q \to \infty}\int_{\mb{S}^+} \theta_{N+1}^{1-2s} \la_{p_q}^{-\gamma}U(\la_{p_q}\theta) Z_{\alpha,k,n,j}(\theta) \, dS\\
=\sum_{i=1}^{M_{\alpha,k,n}} \beta_i\int_{\mb{S}^+} \theta_{N+1}^{1-2s} Z_{\alpha,k,n,i} Z_{\alpha,k,n,j}\, dS =\beta_j,
\end{multline}
for any $j \in  \{1,\dots,M_{\alpha,k,n}\}$.
On the other hand for any $r \in (0,r_0]$
\begin{multline}\label{dim_blow_up_extension_precise:3}
\lim_{q \to \infty}\la_{p_q}^{-\gamma}\varphi_{n,j}(\la_{p_q})=\frac{\varphi_{n,j}(r)}{r^{\gamma}}+\frac{\gamma r^{-N+2s-2\gamma}}{N-2s+2\gamma}\int_{0}^{r}\rho^{-1+\rho}\Upsilon_{n,j}(\rho) d\rho\\
+\frac{N-2s+\gamma}{N-2s+2 \gamma}\int_{0}^{r}\rho^{-N-1+2s-\gamma} \Upsilon_{n,j}(\rho )\, d\rho
\end{multline}
by \eqref{eq_asymptotic_fourier}. Hence 
\begin{multline}\label{dim_blow_up_extension_precise:4}
\beta_j=\frac{\varphi_{n,j}(r)}{r^{\gamma}}+\frac{\gamma r^{-N+2s-2\gamma}}{N-2s+2\gamma}\int_{0}^{r}\rho^{-1+\rho}\Upsilon_{n,j}(\rho) d\rho\\
+\frac{N-2s+\gamma}{N-2s+2 \gamma}\int_{0}^{r}\rho^{-N-1+2s-\gamma} \Upsilon_{n,j}(\rho )\, d\rho
\end{multline}
for any $j \in  \{1,\dots,M_{\alpha,k,n}\}$ and in particular $\beta_j$ does not depend  neither 
on the sequence $\la_p \to 0^+$ nor on its subsequence $\la_{p_q}\to 0^+$. Then by \eqref{dim_blow_up_extension_precise:4} and the Urysohn Subsequence Principle we conclude that \eqref{eq_limit_extension_blow_up_precise} holds, thus completing the proof.
\end{proof}

From Theorem \ref{theorem_blow_up_extension_precise}, Proposition \ref{prop_trace}  and Remark \ref{remark_trace} we can easily deduce the following theorem.

\begin{theorem}\label{theorem_blow_up_precise}
Let $u$ be a solution of \eqref{eq_weak_formulation_hardy_operator} and suppose that $g$ satisfies \eqref{hp_g}. Let $\gamma$, $n \in \mathbb{N} \setminus\{0\}$, $M_{\alpha,k,n}$ and $\{Z_{\alpha,k,n,i}\}_{i \in \{1,\dots.M_{\alpha,k,n}\}}$ be as in Theorem \ref{theorem_blow_up_extension_precise}. Then 
\begin{equation}\label{eq_limit_blow_up_precise}
\frac{u(\la x)}{\la^\gamma} \to |x|^\gamma\sum_{i=1}^{M_{\alpha,k,n}} \beta_i \Tr(Z_{\alpha,k,n,i}((\cdot/|\cdot|))(x) \quad \text{ strongly in } H^s(B_1') \text { as } \la \to 0^+,
\end{equation}
where $\beta_i$ is as in \eqref{def_betai} for any $i \in \{1,\cdots,M_{\alpha,k,n}\}$.
\end{theorem}

\begin{proof}[\textbf{Proof of Corollary \ref{corollary_unique_continuation_extention} and Corollary \ref{corollary_unique_continuation}}]
We start by proving Corollary \ref{corollary_unique_continuation_extention}. Let $U$ be a solution of \eqref{eq_extension_with_equation} such that \eqref{eq_corollary_unique_continuation_extention} 
holds and assume by contradiction that $U\not\equiv 0$ on $\Omega \times (0,\infty)$. Let $\gamma$ be as in Theorem \ref{theorem_blow_up_extension_precise}. Then there exists a sequence $\la_q\to 0^+$ 
such that 
\begin{equation}\label{dim_corollaries:1}
\lim_{q \to \infty}\la_q^{-\gamma} U(\la_q z)=0 \quad \text{ for a.e } z \in B_1^+.
\end{equation}
On the other hand by Theorem \ref{theorem_blow_up_extension_not_precise} there exists an eigenfunction $Z$ of \eqref{prob_eigenvalue_sphere} such that 
\begin{equation}\label{dim_corollaries:2} 
\lim_{q \to \infty}\la_q^{-\gamma} U(\la_q z)=|z|^{\gamma}Z(z/|z|) \quad \text{ for a.e. } z \in B_1^+,
\end{equation}
up to a further subsequence, which is a contradiction.
Arguing in the same way, we can deduce Corollary \ref{corollary_unique_continuation} from Theorem \ref{theorem_blow_up_not_precise}, taking into account Remark \ref{reamrk_eigenfunction0}.
\end{proof}

\section{Computation of the first eigenvalue on a hemisphere}\label{sec_computation_eigenvalue}

\begin{proposition} 
Equation   \eqref{eq_first_eigenvlaue} holds for any $k \in \{3, \dots, N\}$. If $k=N$ then \eqref{eq_first_eigenvlaue_k=N} holds.
\end{proposition}
\begin{proof}
Let $Y_{\alpha,k,1}$ be the first eigenfunction of \eqref{prob_eigenvalue_omega} defined in Section \ref{sec_functional_setting_and_main_result}. In particular $Y_{\alpha,k,1}$ is positive.
By  \cite[Theorem 1.1]{FVF} there exists  an eigenfunction  $\Psi$  of problem \eqref{prob_eigenvalue_S'}, corresponding to the first eigenvalue  $\eta_{\alpha,k,1}$, such that 
\begin{equation}\label{dim_first_eigenvalue:1}
\la^{\frac{N-2}{2}-\sqrt{\left(\frac{N-2}{2}\right)^2+\eta_{\alpha,k,1}}}Y_{\alpha,k,1}(\la x) \to |x|^{-\frac{N-2}{2}+\sqrt{\left(\frac{N-2}{2}\right)^2+\eta_{\alpha,k,1}}} 
\Psi\left(\frac{x}{|x|}\right)
\end{equation}
strongly in $H^1(B_1')$ as $\la \to 0^+$, since  $Y_{\alpha,k,1}$ is positive.
Furthermore for any $\phi \in C^\infty_c(\Omega)$
\begin{equation}\label{dim_first_eigenvalue:2}
\df{(\mb{H}_{\alpha,k}^s(\Omega))^*}{L_{\alpha,k}^s Y_{\alpha,k,1}}{\phi}{\mb{H}_{\alpha,k}^s(\Omega)}=
(Y_{\alpha,k,1},\phi)_{\mb{H}_{\alpha,k}^s(\Omega)}= \mu_{\alpha,k,1}^s \int_{\Omega}Y_{\alpha,k,1} \phi \, dx,
\end{equation}
in view of \eqref{fract-lapla-domain-scalar-product}, that  is $Y_{\alpha,k,1}$ is weak solution of $L_{\alpha,k}^s Y_{\alpha,k,1}= \mu_{\alpha,k,1}^sY_{\alpha,k,1}$ 
in the sense given by \eqref{eq_weak_formulation_hardy_operator}. Let $U$ be the extension of $Y_{\alpha,k,1}$ provided by Theorem \ref{theorem_extension}. Since $Y_{\alpha,k,1}$ is positive then $|U|$ 
is the only  solution to the minimization problem \eqref{def_min_prob} and so we conclude that $U$ is positive. Then, in view of by Theorem \ref{theorem_blow_up_extension_precise} and Theorem
\ref{theorem_blow_up_precise},  
\begin{equation}\label{dim_first_eigenvalue:3}
\la^{\frac{N-2s}{2}-\sqrt{\left(\frac{N-2s}{2}\right)^2+\gamma_{\alpha,k,1}}}Y_{\alpha,k,1}(\la x)\to |x|^{-\frac{N-2s}{2}+\sqrt{\left(\frac{N-2s}{2}\right)^2+\gamma_{\alpha,k,1}}} 
\beta_1 \Tr(Z_{\alpha,k,1}((\cdot/|\cdot|))(x) 
\end{equation}
strongly in $H^s(B_1')$ as $\la \to 0^+$. Putting together \eqref{dim_first_eigenvalue:1} and \eqref{dim_first_eigenvalue:3} we obtain 
\begin{equation}\label{dim_first_eigenvalue:4}
-\frac{N-2s}{2}+\sqrt{\left(\frac{N-2s}{2}\right)^2+\gamma_{\alpha,k,1}}=-\frac{N-2}{2}+\sqrt{\left(\frac{N-2}{2}\right)^2+\eta_{\alpha,k,1}}
\end{equation}
thus \eqref{eq_first_eigenvlaue} follows from a direct computation. Finally, if $k=N$, problem \eqref{prob_eigenvalue_S'} reduces to 
\begin{equation}\label{dim_first_eigenvalue:5}
-\Delta_{\mb{S}'}\Psi-\alpha\Psi= \eta \Psi \quad \text{ in } \mb{S}'
\end{equation}
which admits $-\alpha$ as first eigenvalue, hence we have proved \eqref{eq_first_eigenvlaue_k=N} in view of \eqref{eq_first_eigenvlaue}.
\end{proof}

\bigskip\noindent{\bf Acknowledgements.} 
The author would like to thank Prof. Veronica Felli for helpful discussions and  insightful suggestions which helped improve the manuscript.

\appendix
\section{A proof of Proposition \ref{prop_Halphak_Hs}}\label{sec_appendix_A}
In this section we provide, for the sake of completeness, a detailed proof of Proposition \ref{prop_Halphak_Hs} starting with a preliminary lemma. 
Let us consider,  for  any positive sequence  $\{q_n\}_{n \in \mb{N}}$,
the weighted $\ell^2(\mb{N})$-space defined as  
\begin{equation}\label{def_weighted_l2}
\ell^2(\mb{N}, \{q_n\}):=\left\{\{a_n\}_{n \in \mb{N}}:\sum_{n=0}^\infty q_na_n^2<+\infty\right\}
\end{equation}
endowed with the norm 
\begin{equation}\label{def_weighted_l2_norm}
\norm{\{a_n\}}_{\ell^2(\mb{N}, \{q_n\})}:=\left(\sum_{n=0}^\infty q_na_n^2\right)^\frac{1}{2}.
\end{equation}

\begin{lemma}\label{lemma_interpolation_l2_weighted}
Let  $\ell^2(\mb{N}, \{q_n\})$ and $\ell^2(\mb{N}, \{p_n\})$ be weighted ${\ell}^2(\mb{N})$-spaces. Then 
\begin{equation}\label{eq_interpolation_l2_weighted}
(\ell^2(\mb{N}, \{q_n\}),\ell^2(\mb{N}, \{p_n\}))_{s,2}=\ell^2(\mb{N}, \{q_n^{1-s}p_n^s\}).
\end{equation}
with equivalent norms.
\end{lemma}

\begin{proof}
We follow the proof of \cite[Lemma 23.1]{T}.
Let us consider a variant of the standard $K$ function defined as 
\begin{equation}\label{dim_interpolation_l2_weighted:1}
K_2(t,a):=\inf_{b+c=a}\left\{\left(\norm{b}^2_{\ell^2(\mb{N}, \{q_n\})}+t^2\norm{c}^2_{\ell^2(\mb{N}, \{p_n\})}\right)^\frac{1}{2}:b \in \ell^2(\mb{N}, \{q_n\}), c \in \ell^2(\mb{N}, \{p_n\})\right\},
\end{equation}
for any $t\ge 0$ and any sequence $a \in \ell^2(\mb{N}, \{q_n\})+ \ell^2(\mb{N}, \{p_n\})$. If $K(t,a)$ is the standard $K$-function it is clear that 
$K_2(t,a) \le K(t,a) \le \sqrt{2} K_2(t,a)$ for any  $t\ge 0$ and any sequence $a \in \ell^2(\mb{N}, \{q_n\})+ \ell^2(\mb{N}, \{p_n\})$. It follows that  we can use $K_2$ to define a  norm on 
$(\ell^2(\mb{N}, \{q_n\}),\ell^2(\mb{N}, \{p_n\}))_{s,2}$ equivalent to the standard one. 

We can compute $K_2(a,t)$ explicitly. Indeed, fixed $a \in \ell^2(\mb{N}, \{q_n\})+ \ell^2(\mb{N}, \{p_n\})$  and $t \ge 0$, we can, for any $n \in  \mb{N}$, minimize the 
value of $b_n^2 q_n+t^2(a_n-b_n)^2 p_n$ as a function of $b_n$ choosing 
\begin{equation}\label{dim_interpolation_l2_weighted:2}
b_n:=\frac{t^2 p_n}{q_n+t^2p_n}a_n.
\end{equation}
With this optimal choice it follows that 
\begin{equation}\label{dim_interpolation_l2_weighted:3}
c_n=a_n-b_n=\frac{q_n}{q_n+t^2p_n}a_n
\end{equation}
and  so we obtain
\begin{equation}\label{dim_interpolation_l2_weighted:4}
K_2(t,a)^2=\sum_{n=0}^\infty \frac{t^2 p_nq_n}{q_n+t^2p_n}a_n^2.
\end{equation}
Then by the Monotone Convergence Theorem and the change of variables $t= \tau \sqrt{\frac{q_n}{p_n}}$
\begin{equation}\label{dim_interpolation_l2_weighted:5}
\int_{0}^{\infty}	K_2(t,a)^2 t^{-1-2s} \, dt =\sum_{n=0}^\infty a_n^2 \int_{0}^{\infty} \frac{t^{1-2s} p_nq_n}{q_n+t^2p_n} \, dt
=\left(\int_{0}^{\infty} \frac{\tau^{1-2s} }{1+\tau^2} \, d\tau\right)\sum_{n=0}^\infty a_n^2 q_n^{1-s}p_n^s.
\end{equation}
Since for any $ s \in (0,1)$
\begin{equation}\label{dim_interpolation_l2_weighted:6}
\int_{0}^{\infty} \frac{\tau^{1-2s} }{1+\tau^2} \, d\tau < + \infty, 
\end{equation}
we conclude that \eqref{eq_interpolation_l2_weighted} holds.
\end{proof}

\begin{proof}[\textbf{Proof of Proposition \ref{prop_Halphak_Hs}}.]
Let us start by proving that for any $k \in \{3,\dots,N\}$ and $\alpha$ as in \eqref{hp_alpha} 
\begin{equation}\label{dim_Halphak_Hs:1}
\mb{H}^1_{\alpha,k}(\Omega):=\left\{v \in L^2(\Omega):	\sum_{n=1}^\infty\mu_{\alpha,k,n}v_n^2	<+\infty\right\}=H^1_0(\Omega),
\end{equation}
with equivalent norms. If $u \in  H^1_0(\Omega)$ then, in view of Remark \ref{remark_equivalence_norm},
\begin{equation}\label{dim_Halphak_Hs:2}
u=\sum_{n=1}^{\infty}\left(u, \frac{Y_{\alpha,k,n}}{\sqrt{\mu_{\alpha,k,n}}}\right)_{\alpha,k} \, \frac{Y_{\alpha,k,n}}{\sqrt{\mu_{\alpha,k,n}}}
\end{equation}
and so by the Parseval's identity, \eqref{eq_eigenvalue_omega}, \eqref{def_vn} and  Remark \ref{remark_equivalence_norm}
\begin{equation}\label{dim_Halphak_Hs:3}
+\infty >\norm{u}_{\alpha,k}^2=\sum_{n=1}^\infty \mu_{\alpha,k,n} u_n^2.
\end{equation}                                                                           
On the other hand if $u \in \mb{H}^1_{\alpha,k}(\Omega)$ let, in view of \eqref{eq_eigenvalue_omega},
\begin{equation}\label{dim_Halphak_Hs:4}
u^{(j)}:=\sum_{n=1}^{j}\left(u, \frac{Y_{\alpha,k,n}}{\sqrt{\mu_{\alpha,k,n}}}\right)_{\alpha,k}  \frac{Y_{\alpha,k,n}}{\sqrt{\mu_{\alpha,k,n}}}=\sum_{n=1}^{j} u_n Y_{\alpha,k,n}.
\end{equation}
For any $j \in \mb{N} \setminus \{0\}$ it is clear that $u^{(j)} \in H^1_0(\Omega)$ and if $j>i$  
\begin{equation}\label{dim_Halphak_Hs:5}
\norm{u^{(j)}-u^{(i)}}_{\alpha,k}^2=\sum_{n=i}^j \mu_{\alpha,k,n}u_n^2.
\end{equation}
It follows that  $\{u^{(j)}\}_{j \in \mb{N}\setminus\{0\}}$ converges to $u$ in $H^1_0(\Omega)$ by Remark \ref{remark_equivalence_norm}, and  \eqref{dim_Halphak_Hs:5}. In conclusion $u \in 
H^1_0(\Omega)$. From Remark \ref{remark_equivalence_norm}  and \eqref{dim_Halphak_Hs:3} we deduce that the norms on $H^1_0(\Omega)$ and $\mb{H}^1_{\alpha,k}(\Omega)$ are equivalent.

For any $s \in (0,1]$, since $L^2(\Omega)$ and $\mb{H}^s_{\alpha,k}(\Omega)$ are isomorphic to  $\ell^2(\mb{N})$ and $\ell^2(\mb{N}, \{\mu_{\alpha,k,n}^s\})$  
respectively, from Lemma \ref{lemma_interpolation_l2_weighted} and \eqref{dim_Halphak_Hs:1} it follows that 
\begin{equation}\label{dim_Halphak_Hs:6}
\mb{H}_{\alpha,k}^s(\Omega)=(L^2(\Omega),\mb{H}^1_{\alpha,k}(\Omega))_{s,2}=(L^2(\Omega),H^1_0(\Omega))_{s,2}=
\begin{cases}
H^s_0(\Omega), &\text{if } s \in(0,1)\setminus\{\frac{1}{2}\}, \\
H_{00}^{1/2}(\Omega), &\text{if } s =\frac{1}{2},
\end{cases}	 
\end{equation}
with equivalent norms. The last equality is a  classical interpolation result, see for example \cite{LM1}.
\end{proof}

\bibliographystyle{acm}
\bibliography{Fractional_power_of_Laplacian_with_Hardy_potential_revisited}	
\end{document}